\documentclass[11pt, oneside]{article}

\usepackage{geometry}
\usepackage{amsmath,graphicx,amssymb,amsthm,amstext}
\usepackage{caption}
\usepackage{subcaption}
\usepackage{tikz}          		

\newtheorem{theorem}{Theorem}[section]
\newtheorem{lemma}[theorem]{Lemma}

\newenvironment{definition}[1][Definition]{\begin{trivlist}
\item[\hskip \labelsep {\bfseries #1}]}{\end{trivlist}}

\title{Extremal Numbers for $2 \rightarrow 1$ Directed Hypergraphs with Two Edges Part II: The Degenerate Cases}
\author{Alex Cameron}

\begin{document}

\maketitle

\begin{abstract}
Let a $2 \rightarrow 1$ directed hypergraph be a 3-uniform hypergraph where every edge has two tail vertices and one head vertex. For any such directed hypergraph $F$ let the $n$th extremal number of $F$ be the maximum number of edges that any directed hypergraph on $n$ vertices can have without containing a copy of $F$. There are actually two versions the directed hypergraph model for this problem: the standard version where every triple of vertices is allowed to have up to all three possible directed edges and the oriented version where each triple can have at most one directed edge. In this paper, we determine the standard extremal numbers and the oriented extremal numbers for four different directed hypergraphs. Each has exactly two edges, and of the seven (nontrivial) ($2 \rightarrow 1$)-graphs with exactly two edges, these are the only four with extremal numbers that are quadratic in $n$. The standard and oriented extremal numbers for the other three directed hypergraphs with two edges are determined in a companion paper \cite{cameron2015}.
\end{abstract}

\section{Introduction}

The combinatorial structure treated in this paper is a $2 \rightarrow 1$ directed hypergraph defined as follows.

\begin{definition}
A \emph{$2 \rightarrow 1$ directed hypergraph} is a pair $H = (V,E)$ where $V$ is a finite set of \emph{vertices} and the set of \emph{edges} $E$ is some subset of the set of all pointed $3$-subsets of $V$. That is, each edge is three distinct elements of $V$ with one marked as special. This special vertex can be thought of as the \emph{head} vertex of the edge while the other two make up the \emph{tail set} of the edge. If $H$ is such that every $3$-subset of V contains at most one edge of $E$, then we call $H$ \emph{oriented}. For a given $H$ we will typically write its vertex and edge sets as $V(H)$ and $E(H)$. We will write an edge as $ab \rightarrow c$ when the underlying $3$-set is $\{a,b,c\}$ and the head vertex is $c$.
\end{definition}

For simplicity from this point on we will always refer to $2 \rightarrow 1$ directed hypergraphs as just \emph{graphs} or sometimes as \emph{$(2 \rightarrow 1)$-graphs} when needed to avoid confusion. This structure comes up as a particular instance of the model used to represent definite Horn formulas in the study of propositional logic and knowledge representation \cite{angluin1992, russell2002}. Some combinatorial properties of this model have been recently studied by Langlois, Mubayi, Sloan, and Gy. Tur\'{a}n in \cite{langlois2009} and \cite{langlois2010}. In particular, they looked at the extremal numbers for a couple of different small graphs. Before we can discuss their results we will need the following definitions.

\begin{definition}
Given two graphs $H$ and $G$, we call a function $\phi:V(H) \rightarrow V(G)$ a homomorphism if it preserves the edges of $H$: \[ab \rightarrow c \in E(H) \implies \phi(a)\phi(b) \rightarrow \phi(c) \in E(G).\] We will write $\phi:H \rightarrow G$ to indicate that $\phi$ is a homomorphism.
\end{definition}

\begin{definition}
Given a family $\mathcal{F}$ of graphs, we say that a graph $G$ is \emph{$\mathcal{F}$-free} if no injective homomorphism $\phi:F \rightarrow G$ exists for any $F \in \mathcal{F}$. If $\mathcal{F} = \{F\}$ we will write that $G$ is $F$-free.
\end{definition}

\begin{definition}
Given a family $\mathcal{F}$ of graphs, let the \emph{$n$th extremal number} $\text{ex}(n,\mathcal{F})$ denote the maximum number of edges that any $\mathcal{F}$-free graph on $n$ vertices can have. Similarly, let the \emph{$n$th oriented extremal number} $\text{ex}_o(n,\mathcal{F})$ be the maximum number of edges that any $\mathcal{F}$-free oriented graph on $n$ vertices can have. Sometimes we will call the extremal number the \emph{standard} extremal number or refer to the problem of determining the extremal number as the \emph{standard version} of the problem to distinguish these concepts from their oriented counterparts. As before, if $\mathcal{F} = \{F\}$, then we will write $\text{ex}(n,F)$ or $\text{ex}_o(n,F)$ for simplicity.
\end{definition}

These are often called Tur\'{a}n-type extremal problems after Paul Tur\'{a}n due to his important early results and conjectures concerning forbidden complete $r$-graphs \cite{turan1941, turan1954, turan1961}. Tur\'{a}n problems for uniform hypergraphs make up a large and well-known area of research in combinatorics, and the questions are often surprisingly difficult.

Extremal problems like this have also been considered for directed graphs and multigraphs (with bounded multiplicity) in \cite{brown1973} and \cite{brown1969} and for the more general directed multi-hypergraphs in \cite{brown1984}. In \cite{brown1969}, Brown and Harary determined the extremal numbers for several types of specific directed graphs. In \cite{brown1973}, Brown, Erd\H{o}s, and Simonovits determined the general structure of extremal sequences for every forbidden family of digraphs analogous to the Tur\'{a}n graphs for simple graphs.

The model of directed hypergraphs studied in \cite{brown1984} have $r$-uniform edges such that the vertices of each edge is given a linear ordering. However, there are many other ways that one could conceivably define a uniform directed hypergraph. The graph theoretic properties of a more general definition of a nonuniform directed hypergraph were studied by Gallo, Longo, Pallottino, and Nguyen in \cite{gallo1993}. There a directed hyperedge was defined to be some subset of vertices with a partition into head vertices and tail vertices.

Recently in \cite{cameron2016}, this author tried to capture many of these possible definitions for ``directed hypergraph" into one umbrella class of relational structures called generalized directed hypergraphs. The structures in this class include the uniform and simple versions of undirected hypergraphs, the totally directed hypergraphs studied in \cite{brown1984}, the directed hypergraphs studied in \cite{gallo1993}, and the $2 \rightarrow 1$ model studied here and in \cite{langlois2009,langlois2010}.

In \cite{langlois2009, langlois2010}, they study the extremal numbers for two different graphs with two edges each. They refer to these two graphs as the 4-resolvent and the 3-resolvent configurations after their relevance in propositional logic. Here, we will denote these graphs as $R_4$ and $R_3$ respectively and define them formally as \[V(R_4) = \{a,b,c,d,e\} \text{ and } E(R_4) = \{ab \rightarrow c, cd \rightarrow e\}\] and \[V(R_3) = \{a,b,c,d\} \text{ and } E(R_3) = \{ab \rightarrow c, bc \rightarrow d\}.\]

In \cite{langlois2010} the authors determined $\text{ex}(n,R_4)$ for sufficiently large $n$, and in \cite{langlois2009} they determined a sequence of numbers asymptotically equivalent to the sequence of numbers $\text{ex}(n,R_3)$ as $n$ increases to infinity. In these papers, the authors discuss a third graph with two edges which they call an Escher configuration because it calls to mind the Escher piece where two hands draw each other. This graph is on four vertices, $\{a,b,c,d\}$ and has edge set $\{ab \rightarrow c,cd \rightarrow b\}$. We will denote it by $E$. These three graphs along with the graph made up of two completely overlapping edges on the same 3-set actually turn out to be the only four nondegenerate graphs with exactly two edges. Their standard and oriented extremal numbers are shown in \cite{cameron2015}.

\begin{definition}
A graph $H$ is \emph{degenerate} if its vertices can be partitioned into three sets, $V(H) = T_1 \cup T_2 \cup K$ such that every edge of $E(H)$ is of the form $t_1t_2 \rightarrow k$ for some $t_1 \in T_1$, $t_2 \in T_2$, and $k \in K$.
\end{definition}

An immediate consequence of a result shown in \cite{cameron2016} is that the extremal numbers for a graph $H$ are cubic in $n$ if and only if $H$ is not degenerate.

In our model of directed hypergraph, there are nine different graphs with exactly two edges. Of these, five are degenerate. One of these is the graph with two independent edges, $V=\{a,b,c,d,e,f\}$ and $E=\{ab \rightarrow c, de \rightarrow f\}$. In this case the extremal numbers come directly from the known extremal number for two independent edges for undirected $3$-graphs. Therefore, the oriented extremal number is ${n-1 \choose 2}$ and the standard extremal number is $3{n-1 \choose 2}$.

We will call the other four degenerate graphs with two edges $I_0$, $I_1$, $H_1$, and $H_2$ and define them as
\begin{itemize}
\item $V(I_0) = \{a,b,c,d,x\} \text{ and } E(I_0) = \{ab \rightarrow x, cd \rightarrow x\}$
\item $V(I_1) = \{a,b,c,d\} \text{ and } E(I_1) = \{ab \rightarrow c, ad \rightarrow c\}$
\item $V(H_1) = \{a,b,c,d,x\} \text{ and } E(H_1) = \{ax \rightarrow b, cx \rightarrow d\}$
\item $V(H_2) = \{a,b,c,d\} \text{ and } E(H_2) = \{ab \rightarrow c, ab \rightarrow d\}$
\end{itemize}

Here, the subscripts indicate the number of tail vertices common to both edges. The $I$ graphs also share a head vertex while the $H$ graphs do not.

Some of the proofs that follow rely heavily on the concept of a link graph. For undirected $r$-graphs, the link graph of a vertex is the $(r-1)$-graph induced on the remaining vertices such that each $(r-1)$-set is an $(r-1)$-edge if and only if that set together with the specified vertex makes an $r$-edge in the original $r$-graph \cite{keevash2011}. In the directed hypergraph model here, there are a few ways we could define the link graph of a vertex. We will need the following two.

\begin{definition}
Let $x \in V(H)$ for some graph $H$. The \emph{tail link graph} of $x$ $T_x$ is the simple undirected 2-graph on the other $n-1$ vertices of $V(H)$ with edge set defined by all pairs of vertices that exist as tails pointing to $x$ in some edge of $H$. That is, $V(T_x) = V(H) \setminus \{x\}$ and \[ E(T_x) = \{yz : yz \rightarrow x \in H\}.\] The size of this set, $|T_x|$ will be called the \emph{tail degree} of $x$. The degree of a particular vertex $y$ in the tail link graph of $x$ will be denoted $d_x(y)$.

Similarly, let $D_x$ be the \emph{directed link graph} of $x$ on the remaining $n-1$ vertices of $V(H)$. That is, let $V(D_x) = V(H) \setminus \{x\}$ and \[E(D_x) = \{y \rightarrow z : xy \rightarrow z \in E(H)\}.\]
\end{definition}

The following notation will also be used when we want to count edges by tail sets.

\begin{definition}
For any pair of vertices $x,y \in V(H)$ for some graph $H$ let $t(x,y)$ denote the number of edges with tail set $\{x,y\}$. That is \[t(x,y) = |\{v : xy \rightarrow v \in E(H)\}|.\]
\end{definition}

\section{Forbidden $I_0$}

In this section $I_0$ denotes the forbidden graph where two edges intersect in exactly one vertex such that this vertex is a head for both edges. That is $V(I_0) = \{a,b,c,d,x\}$ and $E(I_0) = \{ab \rightarrow x, cd \rightarrow x\}$ (see Figure~\ref{A}). In this section we will prove the following result on the oriented extremal numbers of $I_0$.

\begin{figure}
	\centering
	\begin{tikzpicture}
		\filldraw[black] (0,0) circle (1pt);
		\filldraw[black] (4,0) circle (1pt);
		\filldraw[black] (2,1) circle (1pt);
		\filldraw[black] (0,2) circle (1pt);
		\filldraw[black] (4,2) circle (1pt);
		
		\draw[thick] (0,0) -- (0,2);
		\draw[thick] (4,0) -- (4,2);
		\draw[thick,->] (0,1) -- (2,1);
		\draw[thick,->] (4,1) -- (2,1);
		
		\node[left] at (0,2) {$a$};
		\node[left] at (0,0) {$b$};
		\node[above] at (2,1) {$x$};
		\node[right] at (4,2) {$c$};
		\node[right] at (4,0) {$d$};
	\end{tikzpicture}
	\caption{$I_0$}
	\label{A}
\end{figure}

\begin{theorem}
\label{exA}
For all $n \geq 9$, \[\text{ex}_o(n,I_0) = \begin{cases} n(n-3) + \frac{n}{3} & n \equiv 0 \text{ mod 3}\\
n(n-3) + \frac{n-4}{3} & n \equiv 1 \text{ mod 3} \\
n(n-3) + \frac{n-5}{3} & n \equiv 2 \text{ mod 3}
\end{cases} \] with exactly one extremal example up to isomorphism when $3|n$, exactly 18 non-isomorphic extremal constructions when $n \equiv 1 \text{ mod 3}$, and exactly 32 constructions when $n \equiv 2 \text{ mod 3}$.
\end{theorem}

The proof for this is rather long. However, in the standard version of the problem where each triple of vertices may hold up to three different directed edges, the problem is much simpler so we will begin there.

\begin{theorem}
For each $n \geq 5$, \[\text{ex}(n,I_0)=n(n-2)\] and for each $n \geq 6$, there are exactly $(n-1)^n$ different labeled $I_0$-free graphs that attain this maximum number of edges.
\end{theorem}

\begin{proof}
Let $H$ be $I_0$-free on $n \geq 5$ vertices. For any $x \in V(H)$, the tail link graph $T_x$ cannot contain two independent edges (see Figure~\ref{Aiff}). Therefore, the edge structure of $T_x$ is either a triangle or a star with $k$ edges all intersecting in a common vertex for some $0 \leq k \leq n-2$. So each vertex $x \in V(H)$ is at the head of at most $n-2$ edges. Hence, \[|E(H)| = \sum_{x \in V(H)} |E(T_x)| \leq n(n-2).\]

\begin{figure}
	\centering
	\begin{tikzpicture}
		\filldraw[color=black,fill=blue!5] (-4,1) circle [radius=2];
		
		\node at (-6,3) {$T_x$};
		\node [below] at (-5,0) {$b$};
		\node [below] at (-3,0) {$d$};
		\node [above] at (-5,2) {$a$};
		\node [above] at (-3,2) {$c$};
		
		\filldraw[black] (-5,0) circle (1pt);
		\filldraw[black] (-3,0) circle (1pt);
		\filldraw[black] (-5,2) circle (1pt);
		\filldraw[black] (-3,2) circle (1pt);
		
		\draw[thick] (-5,0) -- (-5,2);
		\draw[thick] (-3,0) -- (-3,2);
	
		\node at (-1,1) {$\iff$};
		
		\node at (0,3) {$H$};
		\node [below] at (0,0) {$b$};
		\node [below] at (3,0) {$d$};
		\node [above] at (0,2) {$a$};
		\node [above] at (3,2) {$c$};
		\node [above] at (1.5,1) {$x$};

		\filldraw[black] (0,0) circle (1pt);
		\filldraw[black] (3,0) circle (1pt);
		\filldraw[black] (1.5,1) circle (1pt);
		\filldraw[black] (0,2) circle (1pt);
		\filldraw[black] (3,2) circle (1pt);
		
		\draw[thick] (0,0) -- (0,2);
		\draw[thick] (3,0) -- (3,2);
		\draw[thick,->] (0,1) -- (1.5,1);
		\draw[thick,->] (3,1) -- (1.5,1);
	\end{tikzpicture}
	\caption{$ab, cd \in E(T_x)$ if and only if $ab \rightarrow x, cd \rightarrow x \in H$}
	\label{Aiff}
\end{figure}
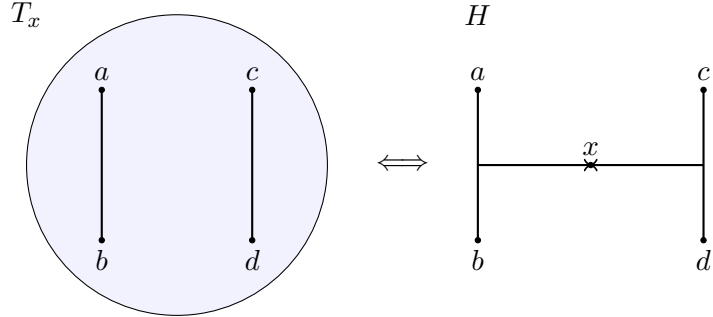

On the other hand, many different extremal constructions exist that give $n(n-2)$ edges on $n$ vertices without the forbidden intersection. Let \[f:[n] \rightarrow [n]\]be any function such that $f(x) \neq x$ for any $x \in [n]$. Define $H_f$ as the graph with vertex set $V(H_f) = [n]$ and edge set \[E(H_f) = \bigcup_{x \in [n]} \left\{f(x)y \rightarrow x : y \in [n] \setminus \{x,f(x)\}\right\}.\]Certainly each vertex $x$ is at the head of $n-2$ edges and each of its tails contain $f(x)$ which prevents the forbidden subgraph. So $|E(H_f)| = n(n-2)$, and $H_f$ is $I_0$-free for any such function $f$.

Moreover, there are $(n-1)^n$ different functions $f$ that will make such a construction on $[n]$. So this gives us $(n-1)^n$ labeledextremal $I_0$-free graphs. Conversely, since any $I_0$-free graph with the maximum number of edges must have $n-2$ edges in $T_x$ for each vertex $x$, then when $n \geq 6$ this implies that all tail link graphs must be $(n-2)$-stars. Therefore, these constructions give all of the extremal examples.
\end{proof}

The oriented version of this problem is less straight forward, but determining $\text{ex}_o(n,I_0)$ also begins with the observation that every tail link graph of an $I_0$-free graph will either be a triangle, a star, or empty. Broadly speaking, as $n$ gets large, it would make more sense for most, if not all, tail link graphs to be stars in order to fit as many edges into an $I_0$-free graph. This motivates the following auxiliary structure.

\subsection{Gates}

Let $H$ be some $I_0$-free graph. For each $x \in V(H)$ for which $T_x$ is a star (with at least one edge), let $g(x)$ denote the common vertex for the edges of $T_x$. We will refer to this vertex as the \emph{gatekeeper} of $x$ (in that it is the gatekeeper that any other vertex must pair with in order to ``access" $x$). In the case where $T_x$ contains only a single edge we may choose either of its vertices to serve as the gatekeeper. In this way, we have constructed a partial function, $g: V(H) \nrightarrow V(H)$.

Next, construct a directed $2$-graph $G$ on the vertex set $V(H)$ based on this partial function: \[y \rightarrow x \in E(G) \iff y = g(x).\] We'll call this digraph the \emph{gate} of $H$ (or more properly, $G$ is the gate of $H$ \emph{under} $g$ since $g$ isn't necessarily unique).

\begin{figure}
  \centering
  \begin{tikzpicture}
      \filldraw[color=black,fill=blue!5] (0,0) circle [radius=2];
      \node at (0,0) {$C_k$};
      
      \filldraw[black] (xyz polar cs:angle=30, radius=3) circle (1pt);
      \filldraw[black] (xyz polar cs:angle=40, radius=3) circle (1pt);
      \filldraw[black] (xyz polar cs:angle=25, radius=4) circle (1pt);
      \draw[thick,->] (xyz polar cs:angle=35, radius=2) -- (xyz polar cs:angle=30, radius=3);
      \draw[thick,->] (xyz polar cs:angle=35, radius=2) -- (xyz polar cs:angle=40, radius=3);
      \draw[thick,->] (xyz polar cs:angle=30, radius=3) -- (xyz polar cs:angle=25, radius=4);
      
      \filldraw[black] (xyz polar cs:angle=340, radius=3) circle (1pt);
      \filldraw[black] (xyz polar cs:angle=345, radius=4) circle (1pt);
      \draw[thick,->] (xyz polar cs:angle=340, radius=2) -- (xyz polar cs:angle=340, radius=3);
      \draw[thick,->] (xyz polar cs:angle=340, radius=3) -- (xyz polar cs:angle=345, radius=4);
      
      \filldraw[black] (xyz polar cs:angle=190, radius=3) circle (1pt);
      \filldraw[black] (xyz polar cs:angle=195, radius=4) circle (1pt);
      \filldraw[black] (xyz polar cs:angle=185, radius=4) circle (1pt);
      \filldraw[black] (xyz polar cs:angle=180, radius=5) circle (1pt);
      \filldraw[black] (xyz polar cs:angle=190, radius=5) circle (1pt);
      \draw[thick,->] (xyz polar cs:angle=190, radius=2) -- (xyz polar cs:angle=190, radius=3);
      \draw[thick,->] (xyz polar cs:angle=190, radius=3) -- (xyz polar cs:angle=195, radius=4);
      \draw[thick,->] (xyz polar cs:angle=190, radius=3) -- (xyz polar cs:angle=185, radius=4);
      \draw[thick,->] (xyz polar cs:angle=185, radius=4) -- (xyz polar cs:angle=190, radius=5);
      \draw[thick,->] (xyz polar cs:angle=185, radius=4) -- (xyz polar cs:angle=180, radius=5);
      
      \filldraw[black] (xyz polar cs:angle=95, radius=3) circle (1pt);
      \filldraw[black] (xyz polar cs:angle=105, radius=3) circle (1pt);
      \filldraw[black] (xyz polar cs:angle=115, radius=3) circle (1pt);
       \draw[thick,->] (xyz polar cs:angle=105, radius=2) -- (xyz polar cs:angle=115, radius=3);
        \draw[thick,->] (xyz polar cs:angle=105, radius=2) -- (xyz polar cs:angle=105, radius=3);
         \draw[thick,->] (xyz polar cs:angle=105, radius=2) -- (xyz polar cs:angle=95, radius=3);
  \end{tikzpicture}
  \caption{The structure of a connected component of the gate $G$}
  \label{Ckplus}
\end{figure}
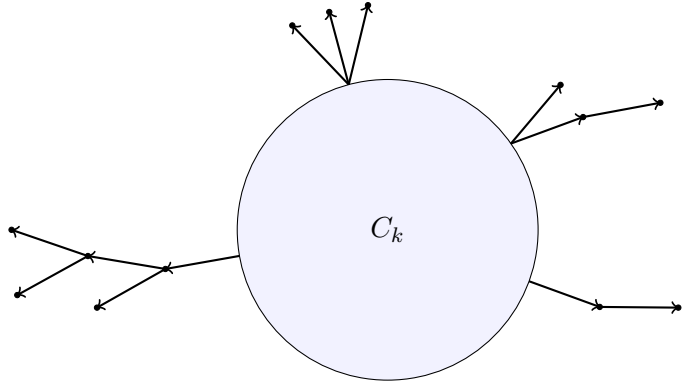

The edge structure of any gate $G$ is not difficult to determine. Since $g$ is a partial function, then each vertex has in-degree at most one in $G$. Therefore, the structure of any connected component of $G$ can be described as a directed cycle on $k$ vertices, $C_k$, for $1 \leq k$ (where $k=1$ implies a single vertex) unioned with $k$ disjoint directed trees, each with its root vertex on this cycle (see Figure~\ref{Ckplus}). We will refer to this kind of general structure as a \emph{$k$-cycle with branches}.

Let \[\mathcal{C} = \bigcup_{k=1}^{n} \mathcal{C}_k\] be the set of maximal connected components of a gate of $H$ where, for each $k$, $\mathcal{C}_k$ is the set of maximal connected components that are $k$-cycles with branches. Note that \[|E(H)| = \sum_{x \in V(H)} |T_x| = \sum_{C \in \mathcal{C}} \left( \sum_{x \in V(C)} |T_x| \right) = \sum_{k=1}^n \left(\sum_{C \in \mathcal{C}_k} \left( \sum_{x \in V(C)} |T_x| \right)\right).\] The next section determines for each $k$ an upper bound on \[\sum_{x \in V(C)} |T_x|\] as a function of the number of vertices, $|V(C)|$, for any $C \in \mathcal{C}_k$.

\subsection{Bounding $\sum_{x \in V(C)} |T_x|$ for any connected component $C$ of the gate}

Loosely speaking, each gatekeeper edge of a connected component $C$ represents at most $n-2$ edges of $H$. We will arrive at an upper bound on the sum $\sum_{x \in V(C)} |T_x|$ by adding this maximum for each edge of $C$, and then subtracting the number of triples of vertices that such a count has included more than once. This will happen for any triple of vertices which contain two or three gatekeeper edges. We make this observation formal in the following definition and lemma.

\begin{definition}
Let $G$ be some gate and let $C$ be a maximal connected component of $G$. Let $P(C)$ be the set of $2 \rightarrow 1$ \emph{possible edges} defined by \[P(C) = \bigcup_{a \rightarrow b \in E(C)} \left\{av \rightarrow b : v \in V(H) \setminus \{a,b\} \right\}.\]
\end{definition}

\begin{lemma}
\label{goofy}
Let $G$ be a gate, and let $C$ be a maximal connected component of $G$. If a set of three distinct vertices $\{x,y,z\} \subseteq V(C)$ are spanned by two gatekeeper edges of $G$, then $P(C)$ contains at least two edges on these three vertices.
\end{lemma}

\begin{proof}
Without loss of generality, the two spanning edges on $\{x,y,z\}$ are either of the form \[x \rightarrow y \rightarrow z \text{ or } x \leftarrow y \rightarrow z.\] In the former case, $P(C)$ contains the edges $xz \rightarrow y$ and $yx \rightarrow z$. In the latter case, $P(C)$ contains the edges $yz \rightarrow x$ and $yx \rightarrow z$.
\end{proof}

Now comes the main counting lemma.

\begin{lemma}
\label{components}
Let $H$ be an $I_0$-free graph on $n \geq 8$ vertices. Let $G$ be a gate of $H$. Let $C$ be a maximal connected component of $G$ with $m$ vertices. Then
\begin{itemize}
\item $\sum_{x \in V(C)} |T_x| \leq m(n-3)$ if $C \in \mathcal{C}_k$ for any $k \neq 3$ with equality possible only if $C = C_k$ for some $k \geq 4$,
\item $\sum_{x \in V(C)} |T_x| \leq m(n-3) + 1$ if $C = C_3$, and
\item $\sum_{x \in V(C)} |T_x| \leq m(n-3)$ for all other $C \in \mathcal{C}_3$ with equality possible only if $C$ is a $3$-cycle with exactly one nonempty directed path coming off of it.
\end{itemize}
\end{lemma}

\begin{proof}
For convenience let \[S = \sum_{x \in V(C)} |T_x|.\] Note that for each $x \in V(C)$ with in-degree one, $ab \in T_x$ implies that $ab \rightarrow x \in P(C)$. Hence, if $C \not \in \mathcal{C}_1$, then every edge counted in the sum $S$ is in $P(C)$. Moreover, $|P(C)| = m(n-2)$.

If $C \in \mathcal{C}_k$ for $k \geq 4$, then by Lemma~\ref{goofy}, each intersection of gatekeeper edges of $C$ yields two edges on the same triple of vertices in $P(C)$. Conversely, since $C$ contains no $C_3$, then each distinct triple of vertices contains at most two gatekeeper edges. Therefore, each triple contains at most two edges of $P(C)$. Hence, \[S \leq m(n-2) - \sum_{x \in V(C)} {d_G(x) \choose 2}\] where $d_G(x)$ denotes the total number of vertices incident to $x$ in the gate.

Since $C$ has $m$ edges, then $\sum_{x \in V(C)} d_G(x) = 2m$. So \[S \leq m(n-2) - \sum_{x \in V(C)} {d_G(x) \choose 2} \leq m(n-3)\] by Jensen's Inequality. Moreover, equality happens if and only if $d_G(x) = d_G(y)$ for all $x,y \in V(C)$. Therefore, this inequality is strict for all $C \in \mathcal{C}_k$ unless $C=C_k$.

Similarly, if $C \in \mathcal{C}_2$, then $P(C)$ contains at least $\sum_{x \in V(C)} {d_G(x) \choose 2}$ multiedges for the same reason as before. But here there are an additional $n-2$ edges counted for each triple containing the $C_2$. Also, \[\sum_{x \in V(C)} d_G(x) = 2(m-1).\] Hence, \[S \leq m(n-2) - (n-2) - \sum_{x \in V(C)} {d_G(x) \choose 2} \leq (m-1)(n-2) - m {\frac{2(m-1)}{m} \choose 2}\] by Jensen's Inequality. This is strictly less than $m(n-3)$.

In the acyclic case, Lemma~\ref{goofy} implies that the sum of all $|T_x|$ for each $x \in V(C)$ other than the root vertex is less than or equal to \[(m-1)(n-2) - \sum_{x \in V(C)} {d_G(x) \choose 2}.\] The root vertex itself is the head vertex at most 3 edges in $H$ so Jensen's Inequality gives \[S \leq (m-1)(n-2) - m {\frac{2(m-1)}{m} \choose 2} + 3 < m(n-3)\] for all $n \geq 8$.

Finally, if $C \in \mathcal{C}_3$, then each intersection of gatekeeper edges of $C$ yields two edges on the same triple of vertices in $P(C)$. However, exactly one triple of vertices contains three gatekeeper edges and has three edges in $P(C)$. But the rest have at most two since there is only one triangle in $C$. Therefore, $\sum_{x \in V(C)} {d_G(x) \choose 2}$ counts each triple of vertices that contain more than one gatekeeper edge exactly once except for the triple that makes up the $C_3$ which it counts three times. Since we must subtract off 2 edges in $P(C)$ on these three vertices to eliminate repeated triples, then we must subtract $\sum_{x \in V(C)} {d_G(x) \choose 2}-1$ from $|P(C)|$. Therefore, \[S \leq m(n-2) - \sum_{x \in V(C)} {d_G(x) \choose 2} + 1.\]

So by Jensen's Inequality, \[S \leq m(n-3) + 1\] with equality possible only if all of the degrees $d_G(x)$ are equal. This can only happen if $C = C_3$.

If we want to see for which $C \in \mathcal{C}_3$ the second best bound of $m(n-3)$ could be attained, then we need to set \[\sum_{x \in V(C)} {d_G(x) \choose 2} = m+1.\] Assume that the vertices are $x_1,\ldots,x_m$, and for each $x_i$ let \[d_i = d_G(x_i)-2.\] Then $\sum_{i=1}^m d_i = 0$ and a quick calculation shows that $\sum_{i=1}^m d_i^2 = 2$. Therefore, the only possibility is for some $d_i=1$ and another to equal $-1$ and all the rest must be 0. This corresponds with one vertex degree being 3, another being 1, and all others being 2. The only way that this can happen in a $C_3$ with branches is to have exactly one branch, and that branch must be a directed path.
\end{proof}

This shows that the best we can hope for in terms of the average number of edges per vertex over any connected component of the gate is $n-3 + \frac{1}{3}$, and this could be attained only in the case where the component is a directed triangle with no branches. Otherwise, the average number of edges of a component is at most $n-3$, and this is attainable only if the component is a directed triangle with a single directed path coming off of one of its vertices or a directed $k$-cycle with no branches for some $k \geq 4$.

This is enough for us to establish the upper bound for $\text{ex}_o(n,I_0)$ and to characterize the necessary structure of the gate for any graph attaining this upper bound.

\subsection{Upper Bound on $\text{ex}_o(n,I_0)$}

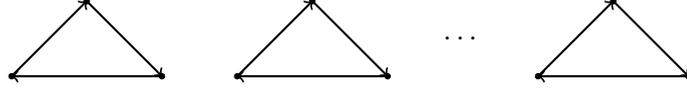
\begin{figure}
  \centering
      \begin{tikzpicture}
      	\filldraw[black] (0,0) circle (1pt);
	\filldraw[black] (1,1) circle (1pt);
	\filldraw[black] (2,0) circle (1pt);
	\draw[thick,->] (0,0) -- (1,1);
	\draw[thick,->] (1,1) -- (2,0);
	\draw[thick,->] (2,0) -- (0,0);
	
	\filldraw[black] (3,0) circle (1pt);
	\filldraw[black] (4,1) circle (1pt);
	\filldraw[black] (5,0) circle (1pt);
	\draw[thick,->] (3,0) -- (4,1);
	\draw[thick,->] (4,1) -- (5,0);
	\draw[thick,->] (5,0) -- (3,0);
	
	\filldraw[black] (7,0) circle (1pt);
	\filldraw[black] (8,1) circle (1pt);
	\filldraw[black] (9,0) circle (1pt);
	\draw[thick,->] (7,0) -- (8,1);
	\draw[thick,->] (8,1) -- (9,0);
	\draw[thick,->] (9,0) -- (7,0);
	
	\node at (6,0.5) {$\cdots$};
      \end{tikzpicture}
  \caption{Structure of the gate for an extremal $I_0$-free graph when $n \equiv 0 \text{ mod } 3$.}
\end{figure}

Let $H$ be an $I_0$-free graph on $n \geq 9$ vertices. Let $G$ be a gate of $H$. Let $\mathcal{C}$ be the set of maximal connected components of $G$ and break $\mathcal{C}$ into three disjoint subsets based on the maximum average number of edges attainable for the components in each. That is, let \[\mathcal{C} = \mathcal{D}_1 \cup \mathcal{D}_2 \cup \mathcal{D}_3\] where $\mathcal{D}_1$ are all components with maximum average number of edges per vertex strictly less than $n-3$: those components that are either acyclic, contain a $C_2$, contain a $C_3$ with nonempty branches that are more than just a single path, or contain a $C_k$ for $k \geq 4$ with some nonempty branch; $\mathcal{D}_2$ is the set of all components with maximum number of edges per vertex of $n-3$: those that contain a directed $C_3$ and exactly one directed path or those that are a directed $k$-cycles for any $k \geq 4$ and no branches; and $\mathcal{D}_3$ is the set of components with a maximum average greater than $n-3$: the directed triangles.

For each $i$ let $d_i$ be the total number of vertices contained in the components of $\mathcal{D}_i$. Then \[|E(H)| \leq d_3 \left(n-3 + \frac{1}{3} \right) + (n-d_3)(n-3)\] with equality possible only if $d_1=0$. Then this is enough to prove the following.

\begin{lemma}
\label{mod0}
Let $H$ be an $I_0$-free graph on $n \geq 9$ vertices such that $n \equiv 0 \text{ mod } 3$, then \[|E(H)| \leq n(n-3) + \frac{n}{3}.\] Moreover, the only way for $H$ to attain this maximum number of edges is if the gate of $H$ is a disjoint union of directed triangles.
\end{lemma}

The next two lemmas give the maximum number when $n \equiv 1,2 \text{ mod } 3$. There is only slightly more to consider in these cases.

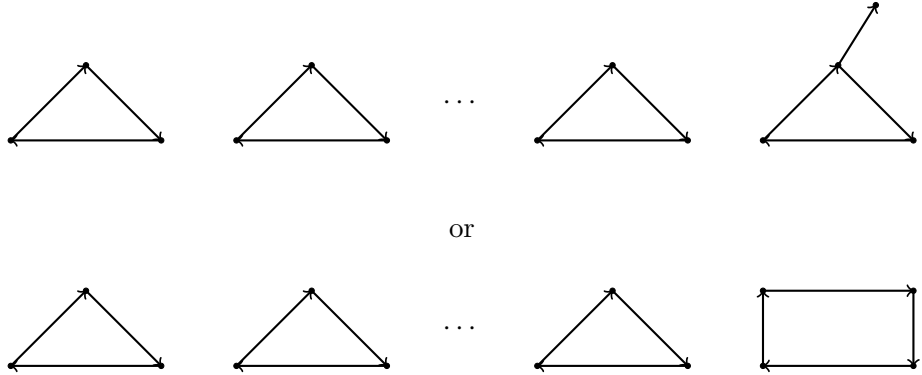
\begin{figure}
  \centering
      \begin{tikzpicture}
      	\filldraw[black] (0,0) circle (1pt);
	\filldraw[black] (1,1) circle (1pt);
	\filldraw[black] (2,0) circle (1pt);
	\draw[thick,->] (0,0) -- (1,1);
	\draw[thick,->] (1,1) -- (2,0);
	\draw[thick,->] (2,0) -- (0,0);
	
	\filldraw[black] (3,0) circle (1pt);
	\filldraw[black] (4,1) circle (1pt);
	\filldraw[black] (5,0) circle (1pt);
	\draw[thick,->] (3,0) -- (4,1);
	\draw[thick,->] (4,1) -- (5,0);
	\draw[thick,->] (5,0) -- (3,0);
	
	\filldraw[black] (7,0) circle (1pt);
	\filldraw[black] (8,1) circle (1pt);
	\filldraw[black] (9,0) circle (1pt);
	\draw[thick,->] (7,0) -- (8,1);
	\draw[thick,->] (8,1) -- (9,0);
	\draw[thick,->] (9,0) -- (7,0);
	
	\filldraw[black] (10,0) circle (1pt);
	\filldraw[black] (11,1) circle (1pt);
	\filldraw[black] (12,0) circle (1pt);
	\filldraw[black] (11.5,1.8) circle (1pt);
	\draw[thick,->] (10,0) -- (11,1);
	\draw[thick,->] (11,1) -- (12,0);
	\draw[thick,->] (12,0) -- (10,0);
	\draw[thick,->] (11,1) -- (11.5,1.8);
	
	\node at (6,0.5) {$\cdots$};
	
	\filldraw[black] (0,-3) circle (1pt);
	\filldraw[black] (1,-2) circle (1pt);
	\filldraw[black] (2,-3) circle (1pt);
	\draw[thick,->] (0,-3) -- (1,-2);
	\draw[thick,->] (1,-2) -- (2,-3);
	\draw[thick,->] (2,-3) -- (0,-3);
	
	\filldraw[black] (3,-3) circle (1pt);
	\filldraw[black] (4,-2) circle (1pt);
	\filldraw[black] (5,-3) circle (1pt);
	\draw[thick,->] (3,-3) -- (4,-2);
	\draw[thick,->] (4,-2) -- (5,-3);
	\draw[thick,->] (5,-3) -- (3,-3);
	
	\filldraw[black] (7,-3) circle (1pt);
	\filldraw[black] (8,-2) circle (1pt);
	\filldraw[black] (9,-3) circle (1pt);
	\draw[thick,->] (7,-3) -- (8,-2);
	\draw[thick,->] (8,-2) -- (9,-3);
	\draw[thick,->] (9,-3) -- (7,-3);
	
	\filldraw[black] (10,-3) circle (1pt);
	\filldraw[black] (10,-2) circle (1pt);
	\filldraw[black] (12,-2) circle (1pt);
	\filldraw[black] (12,-3) circle (1pt);
	\draw[thick,->] (10,-3) -- (10,-2);
	\draw[thick,->] (10,-2) -- (12,-2);
	\draw[thick,->] (12,-2) -- (12,-3);
	\draw[thick,->] (12,-3) -- (10,-3);
	
	\node at (6,-1.2) {or};
	
	\node at (6,-2.5) {$\cdots$};
      \end{tikzpicture}
  \caption{The only possible structures of the gate of an extremal $I_0$-free graph when $n \equiv 1 \text{ mod } 3$.}
\end{figure}

\begin{lemma}
\label{mod1}
Let $H$ be an $I_0$-free graph on $n \geq 9$ vertices such that $n \equiv 1 \text{ mod } 3$, then \[|E(H)| \leq n(n-3) + \frac{n-4}{3}.\] Moreover, the only way for $H$ to attain this maximum number of edges is if the gate of $H$ is a disjoint union of $\frac{n-4}{3}$ directed triangles together with either a directed $C_4$ or a 3-cycle with an extra edge.
\end{lemma}

\begin{proof}
Since $n \equiv 1 \text{ mod } 3$, then $d_3 \leq n-1$. If $d_3 = n-1$, then the gate consists of $\frac{n-1}{3}$ disjoint directed triangles and one isolated vertex which means that \[|E(H)| \leq (n-1) \left(n-3+\frac{1}{3} \right) + 3.\] If $d_3 \leq n-4$, then we can do better with \[|E(H)| \leq (n-4) \left( n-3+\frac{1}{3} \right) + 4(n-3)\] only in the case of $\frac{n-4}{3}$ disjoint directed triangles and one component from $\mathcal{D}_2$ in the gate. Therefore, \[|E(H)| \leq n(n-3) + \frac{n-4}{3}.\]
\end{proof}

\begin{figure}
  \centering
     \begin{tikzpicture}
      	\filldraw[black] (0,0) circle (1pt);
	\filldraw[black] (1,1) circle (1pt);
	\filldraw[black] (2,0) circle (1pt);
	\draw[thick,->] (0,0) -- (1,1);
	\draw[thick,->] (1,1) -- (2,0);
	\draw[thick,->] (2,0) -- (0,0);
	
	\filldraw[black] (3,0) circle (1pt);
	\filldraw[black] (4,1) circle (1pt);
	\filldraw[black] (5,0) circle (1pt);
	\draw[thick,->] (3,0) -- (4,1);
	\draw[thick,->] (4,1) -- (5,0);
	\draw[thick,->] (5,0) -- (3,0);
	
	\filldraw[black] (7,0) circle (1pt);
	\filldraw[black] (8,1) circle (1pt);
	\filldraw[black] (9,0) circle (1pt);
	\draw[thick,->] (7,0) -- (8,1);
	\draw[thick,->] (8,1) -- (9,0);
	\draw[thick,->] (9,0) -- (7,0);
	
	\filldraw[black] (10,0) circle (1pt);
	\filldraw[black] (11,1) circle (1pt);
	\filldraw[black] (12,0) circle (1pt);
	\filldraw[black] (11.5,1.8) circle (1pt);
	\filldraw[black] (12,2.6) circle (1pt);
	\draw[thick,->] (10,0) -- (11,1);
	\draw[thick,->] (11,1) -- (12,0);
	\draw[thick,->] (12,0) -- (10,0);
	\draw[thick,->] (11,1) -- (11.5,1.8);
	\draw[thick,->] (11.5,1.8) -- (12,2.6);
	
	\node at (6,0.5) {$\cdots$};
	
	\filldraw[black] (0,-3) circle (1pt);
	\filldraw[black] (1,-2) circle (1pt);
	\filldraw[black] (2,-3) circle (1pt);
	\draw[thick,->] (0,-3) -- (1,-2);
	\draw[thick,->] (1,-2) -- (2,-3);
	\draw[thick,->] (2,-3) -- (0,-3);
	
	\filldraw[black] (3,-3) circle (1pt);
	\filldraw[black] (4,-2) circle (1pt);
	\filldraw[black] (5,-3) circle (1pt);
	\draw[thick,->] (3,-3) -- (4,-2);
	\draw[thick,->] (4,-2) -- (5,-3);
	\draw[thick,->] (5,-3) -- (3,-3);
	
	\filldraw[black] (7,-3) circle (1pt);
	\filldraw[black] (8,-2) circle (1pt);
	\filldraw[black] (9,-3) circle (1pt);
	\draw[thick,->] (7,-3) -- (8,-2);
	\draw[thick,->] (8,-2) -- (9,-3);
	\draw[thick,->] (9,-3) -- (7,-3);
	
	\filldraw[black] (10,-3) circle (1pt);
	\filldraw[black] (10,-2) circle (1pt);
	\filldraw[black] (12,-2) circle (1pt);
	\filldraw[black] (12,-3) circle (1pt);
	\filldraw[black] (11,-1.5) circle (1pt);
	\draw[thick,->] (10,-3) -- (10,-2);
	\draw[thick,->] (10,-2) -- (11,-1.5);
	\draw[thick,->] (11,-1.5) -- (12,-2);
	\draw[thick,->] (12,-2) -- (12,-3);
	\draw[thick,->] (12,-3) -- (10,-3);
	
	\node at (6,-2.5) {$\cdots$};
	
	\node at (6,-1.2) {or};
      \end{tikzpicture}
  \caption{The only possible structures of the gate of an extremal $I_0$-free graph when $n \equiv 2 \text{ mod } 3$.}
  \label{poss2}
\end{figure}
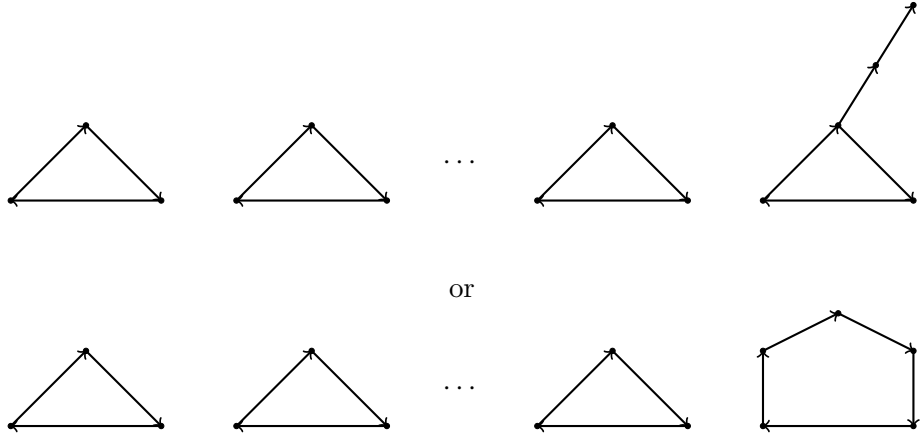

\begin{lemma}
\label{mod2}
Let $H$ be an $I_0$-free graph on $n \geq 11$ vertices such that $n \equiv 2 \text{ mod } 3$, then \[|E(H)| \leq n(n-3) + \frac{n-5}{3}.\] Moreover, the only way for $H$ to attain this maximum number of edges is if the gate of $H$ is a disjoint union of $\frac{n-5}{3}$ directed triangles together with either a directed $C_5$ or a 3-cycle with a directed path of two edges.
\end{lemma}

\begin{proof}
Since $n \equiv 2 \text{ mod } 3$, then $d_3 \leq n-2$ and equality implies that $G$ consists of $\frac{n-2}{3}$ disjoint directed triangles and two additional vertices that are either both isolated, contain one edge, or are a $C_2$ giving $6$, $3+(n-2)$, or $n-2$ additional edges respectively. The best we can do when $d_3=n-2$ is therefore, \[|E(H)| \leq (n-2)\left( n-3+ \frac{1}{3} \right) + (n+1).\]

Otherwise, $d_3 \leq n-5$ and the best we can do is \[|E(H)| \leq (n-5)\left( n-3+ \frac{1}{3} \right) + 5(n-3).\] This is better. Moreover, this will happen only when the five non-triangle vertices are in a component (or components) of $G$ that give an average of $n-3$. So they must either make a $C_5$ or a directed triangle with one path.
\end{proof}

\subsection{Lower bound constructions}

The structure of the gates necessary to attain the maximum number of edges for a $I_0$-free graph determined in the previous section are also sufficient. Of these gates, none of them have acyclic components. Therefore, any graph that produces one of these gates has only vertices with stars for tail link graphs. This immediately implies that there is no $I_0$ in any graph that has such a gate.

Moreover, if $H$ is a graph with a gate $G$ that is one of these configurations, then \[E(H) \subseteq \bigcup_{C \in \mathcal{C}} P(C)\] where $\mathcal{C}$ is the set of maximal connected components of $G$. All that is left to do in order to construct an extremal example is to pick which edges of each $P(C)$ to delete in order to eliminate triples of vertices with more than one edge.

\begin{lemma}
Let $H$ be an $I_0$-free graph on $n \geq 9$ vertices such that $n \equiv 0 \text{ mod } 3$, then \[|E(H)| \geq n(n-3) + \frac{n}{3}\] and there is exactly one extremal construction up to isomorphism.
\end{lemma}

\begin{proof}
We know from Lemma~\ref{mod0} that the only way $H$ can possibly attain $n(n-3) + \frac{n}{3}$ edges is if its gate is the disjoint union of $\frac{n}{3}$ directed triangles. Therefore, each $P(C_3)$ contains exactly one vertex triple with all three possible edges. So two of these must be deleted for each component in order to arrive at an extremal construction. The three choices for this deletion on each component are all isomorphic to each other. Therefore, there is exactly one extremal construction up to isomorphism.
\end{proof}

\begin{lemma}
Let $H$ be an $I_0$-free graph on $n \geq 9$ vertices such that $n \equiv 1 \text{ mod } 3$, then \[|E(H)| \geq n(n-3) + \frac{n-4}{3}\] and there are exactly 18 extremal constructions up to isomorphism.
\end{lemma}

\begin{proof}
We know from Lemma~\ref{mod1} that if $H$ has $n(n-3) + \frac{n-4}{3}$ edges, then its gate is the disjoint union of $\frac{n-4}{3}$ directed triangles with either a directed $C_4$ or a $C_3$ plus an edge on the remaining 4 vertices. As in the previous proof, there is only one choice up to isomorphism for which edges to delete from each $P(C_3)$. However, this will not be true of the last component on the remaining four vertices.

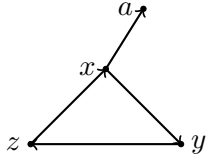
\begin{figure}
  \centering
      \begin{tikzpicture}
      	\filldraw[black] (10,0) circle (1pt);
	\filldraw[black] (11,1) circle (1pt);
	\filldraw[black] (12,0) circle (1pt);
	\filldraw[black] (11.5,1.8) circle (1pt);
	
	\node [left] at (11,1) {$x$};
	\node [left] at (11.5,1.8) {$a$};
	\node [left] at (10,0) {$z$};
	\node [right] at (12,0) {$y$};
	
	\draw[thick,->] (10,0) -- (11,1);
	\draw[thick,->] (11,1) -- (12,0);
	\draw[thick,->] (12,0) -- (10,0);
	\draw[thick,->] (11,1) -- (11.5,1.8);
      \end{tikzpicture}
  \caption{$C_3$ plus an edge}
\end{figure}

First, let's consider the case where the last component is a $C_3$ plus one edge. Call the vertices $\{x,y,z,a\}$ where $x \rightarrow y \rightarrow z \rightarrow x$ is the $C_3$ and $x \rightarrow a$ is the additional edge. First, note that we have the following three mutually exclusive choices for edges with head vertices in this component:
\begin{enumerate}
\item $xa \in T_y$ or $xy \in T_a$,
\item $za \in T_x$ or $xz \in T_a$, and
\item $zx \in T_y$, $yz \in T_x$, or $xy \in T_z$.
\end{enumerate}
This gives 12 choices, and each choice is unique up to isomorphism.

\begin{figure}
  \centering
      \begin{tikzpicture}
      	\filldraw[black] (10,-3) circle (1pt);
	\filldraw[black] (10,-2) circle (1pt);
	\filldraw[black] (12,-2) circle (1pt);
	\filldraw[black] (12,-3) circle (1pt);
	\draw[thick,->] (10,-3) -- (10,-2);
	\draw[thick,->] (10,-2) -- (12,-2);
	\draw[thick,->] (12,-2) -- (12,-3);
	\draw[thick,->] (12,-3) -- (10,-3);
	
	\node [left] at (10, -2.5) {$2$};
	\node [right] at (12, -2.5) {$2$};
      \end{tikzpicture}
  \caption{$C_4$ with 2 additional edges in opposite tail link graphs}
\end{figure}
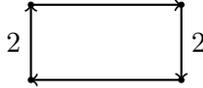

Next consider the case of $C_4$. Each $3$-subset of these four vertices holds two edges of $P(C)$ - one that points along the direction of the two gatekeeper edges and one that points the middle vertex of the two gatekeeper edges. For each triple one of these edges must be deleted to arrive at a legal oriented construction.

Each tail link graph must have at least $n-4$ edges, and combined they must contain four additional edges. Since each can have up to two more edges, then the distribution of these additional edges must be one of the following integer partitions of 4:
\begin{itemize}
\item 2, 2, 0, 0
\item 2, 1, 1, 0
\item 1, 1, 1, 1
\end{itemize}

There is only one choice up to isomorphism with a distribution of 2, 2, 0, 0. Each of the three ways to place 2, 1, 1, 0 around $C_4$ are possible but each distribution has only one way up to isomorphism. Finally, there are two ways up to isomorphism to put an extra edge into each tail link graph. So all together there are six nonisomorphic ways to distribute these extra edges to the $C_4$ tail link graphs.
\end{proof}

\begin{lemma}
Let $H$ be an $I_n$-free graph on $n \geq 9$ vertices such that $n \equiv 2 \text{ mod } 3$, then \[|E(H)| \geq n(n-3) + \frac{n-5}{3}\] and there are exactly 32 extremal constructions up to isomorphism.
\end{lemma}

\begin{proof}

We can do the same kind of analysis when $n = 3k+2$ as in the previous proof. We know from Lemma~\ref{mod2} that the gate of any extremal construction must be all directed triangles together with either a directed $C_5$ or a directed triangle with a directed path of length two coming off of it (see Figure~\ref{poss2}).

First, consider the $C_5$ case. Let the vertices be $\{x_0,\ldots, x_4\}$. For each gatekeeper edge, $x_i \rightarrow x_{i+1}$, every edge of the form $x_iv \rightarrow x_{i+1}$ must be an edge in $H$ for any vertex \[v \neq x_i,x_{i+1},x_{i-1},x_{i+2}.\]

Each gatekeeper edge can represent up to two additional edges of $H$, but again, every intersection of gatekeeper edges requires a mutually exclusive choice. Ultimately, we can add 5 additional edges so the extra edges must be distributed in one of the following ways:
\begin{itemize}
\item 2, 2, 1, 0, 0
\item 2, 1, 1, 1, 0
\item 1, 1, 1, 1, 1
\end{itemize}

There are 2 ways to get the first distribution up to isomorphism, 4 ways to get the second, and 2 ways to get the third. Therefore, there are 8 extremal constructions with this gate up to isomorphism.

Now consider the case of a directed triangle with a directed two path coming off of it. If we label the vertices as $\{x,y,z,a,b\}$ (see Figure~\ref{c3plusplus}), the mutually exclusive choices are
\begin{enumerate}
\item $ax \rightarrow y$ or $yx \rightarrow a$,
\item $az \rightarrow x$ or $zx \rightarrow a$,
\item $zx \rightarrow y$, $yz \rightarrow x$, or $xy \rightarrow z$, and
\item $xa \rightarrow b$ or $bx \rightarrow a$
\end{enumerate}

\begin{figure}
  \centering
      \begin{tikzpicture}
      	\filldraw[black] (10,0) circle (1pt);
	\filldraw[black] (11,1) circle (1pt);
	\filldraw[black] (12,0) circle (1pt);
	\filldraw[black] (11.5,1.8) circle (1pt);
	\filldraw[black] (12,2.6) circle (1pt);
	
	\node[left] at (10,0) {$z$};
	\node[left] at (11,1) {$x$};
	\node[right] at (12,0) {$y$};
	\node[left] at (11.5,1.8) {$a$};
	\node[left] at (12,2.6) {$b$};
	
	\draw[thick,->] (10,0) -- (11,1);
	\draw[thick,->] (11,1) -- (12,0);
	\draw[thick,->] (12,0) -- (10,0);
	\draw[thick,->] (11,1) -- (11.5,1.8);
	\draw[thick,->] (11.5,1.8) -- (12,2.6);
      \end{tikzpicture}
  \caption{$C_3^{+2}$}
  \label{c3plusplus}
\end{figure}

This gives 24 ways of reaching the maximum, and each way is unique up to isomorphism. Therefore, there are 32 total distinct extremal graphs up to isomorphism.
\end{proof}

This establishes the main result of this section.

\section{Forbidden $I_1$}

In this section $I_1$ denotes the forbidden graph where two edges intersect in exactly two vertices such that one vertex is a head for both edges and the other is a tail for each edge. That is $V(I_1) = \{a,b,c,d\}$ and $E(I_1) = \{ab \rightarrow c, ad \rightarrow c\}$ (see Figure~\ref{F}).

\begin{theorem}
For all $n \geq 4$, \[\text{ex}(n,I_1)=\text{ex}_o(n,I_1) = n \left\lfloor \frac{n-1}{2} \right\rfloor\]and there are \[\left(\frac{(n-1)!}{2^{\left\lfloor \frac{n-1}{2} \right\rfloor} \left \lfloor \frac{n-1}{2} \right\rfloor !}\right)^n\] labeled graphs that attain this maximum in the standard case.
\end{theorem}

\begin{figure}
	\centering
	\begin{tikzpicture}
	\filldraw[black] (0,2) circle (1pt);
	\filldraw[black] (1,0) circle (1pt);
	\filldraw[black] (-1,0) circle (1pt);
	\filldraw[black] (0,0.75) circle (1pt);
	
	\draw[thick] (0,2) -- (1,0);
	\draw[thick] (0,2) -- (-1,0);
	\draw[thick,->] (0.5,1) -- (0,0.75);
	\draw[thick,->] (-0.5,1) -- (0,0.75);
	
	\node[above] at (0,2) {$a$};
	\node[below] at (0,0.75) {$c$};
	\node[right] at (1,0) {$d$};
	\node[left] at (-1,0) {$b$};
	
	\end{tikzpicture}
	\caption{$I_1$}
	\label{F}
\end{figure}

\begin{proof}
Let $H$ be an $I_1$-free graph on $n$ vertices. For any $x \in V(H)$, $T_x$ is a simple undirected $2$-graph on $n-1$ vertices such that no two edges are adjacent (this is true for either version of the problem). Therefore, the edges of $T_x$ are a matching on at most $n-1$ vertices. So there are at most $\left\lfloor \frac{n-1}{2} \right\rfloor$ edges in $T_x$ for every $x \in V(H)$. Thus, \[|E(H)| = \sum_{x \in V(H)} |T_x| \leq n \left\lfloor \frac{n-1}{2} \right\rfloor.\] This shows the upper bound for both versions.

Now we want to find lower bound constructions. In the standard version of the problem there are many extremal constructions since for each vertex $x$, we may pick any maximum matching on the remaining $n-1$ vertices to serve as the edges of $T_x$. So \[\text{ex}(n,I_1) = n \left\lfloor \frac{n-1}{2} \right\rfloor.\]

Moreover, the number of labeled graphs that attain this maximum equals the number of ways to take a maximum matching to construct each tail link graph. For even $k$, the number of matchings on $k$ vertices is \[M_k = (k-1)M_{k-2}\] since if we fix some vertex, then we can pick any of the remaining $k-1$ vertices to go with it and then take the number of matchings on the remaining $n-2$. Since $M_2=1$, then in general for even $k$, \[M_k = \prod_{i=1}^{\frac{k}{2}}(2i-1).\]

If $k$ is odd, then we can first select the vertex left out of the matching to get \[M_k = kM_{k-1} = k \cdot \prod_{i=1}^{\frac{k-1}{2}}(2i-1) = \prod_{i=1}^{\frac{k+1}{2}}(2i-1).\] Therefore, the number of labeled extremal $I_1$-free graphs on $n$ vertices is \[\left(\prod_{i=1}^{\left\lfloor\frac{n}{2}\right\rfloor}(2i-1) \right)^n = \left(\frac{(n-1)!}{2^{\left\lfloor \frac{n-1}{2} \right\rfloor} \left \lfloor \frac{n-1}{2} \right\rfloor !}\right)^n.\]

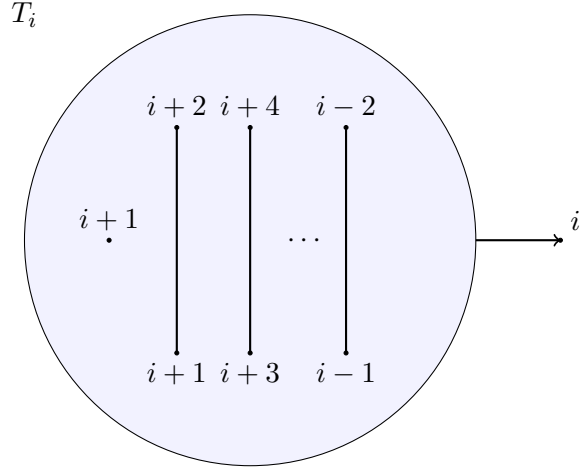
\begin{figure}
  \centering
      \begin{tikzpicture} [scale=0.75]
      \node at (-4,4) {$T_i$};
      
      \filldraw[color=black,fill=blue!5] (0,0) circle [radius=4];
      
      \filldraw[black] (-2.5,0) circle (1pt);
      \node [above] at (-2.5,0) {$i+1$};
      
      \filldraw[black] (-1.3,2) circle (1pt);
      \node [above] at (-1.3,2) {$i+2$};
      \filldraw[black] (0,2) circle (1pt);
      \node [above] at (0,2) {$i+4$};
      \filldraw[black] (1.7,2) circle (1pt);
      \node [above] at (1.7,2) {$i-2$};
      \filldraw[black] (-1.3,-2) circle (1pt);
      \node [below] at (-1.3,-2) {$i+1$};
      \filldraw[black] (0,-2) circle (1pt);
      \node [below] at (0,-2) {$i+3$};
      \filldraw[black] (1.7,-2) circle (1pt);
      \node [below] at (1.7,-2) {$i-1$};
      \draw[thick] (-1.3,2) -- (-1.3,-2);
      \draw[thick] (0,2) -- (0,-2);
      \draw[thick] (1.7,2) -- (1.7,-2);
      
      \node at (1,0) {$\cdots$};
      
      \filldraw[black] (5.5,0) circle (1pt);
      \node [above right] at (5.5,0) {$i$};
      \draw[thick,->] (4,0) -- (5.5,0);
      \end{tikzpicture}
  \caption{$T_i$ in the oriented extremal construction for even $n$}
\end{figure}

In the oriented version of the problem we need to be more careful with the construction. First, assume that $n$ is even and define a graph $H$ with vertex set $V(H) = \mathbb{Z}_n$ and edge set \[E(H) = \bigcup_{i=0}^{n-1} \left\{(i+2k)(i+2k+1) \rightarrow i : k=1,2,\ldots,\frac{n-2}{2} \right\}.\]This construction creates a maximum matching for each tail link graph (with $i+1$ as the odd vertex out for each $T_i$). So $H$ has the extremal number of edges and contains no $I_1$. Therefore, all we need to show is that it has no triple with more than edge.

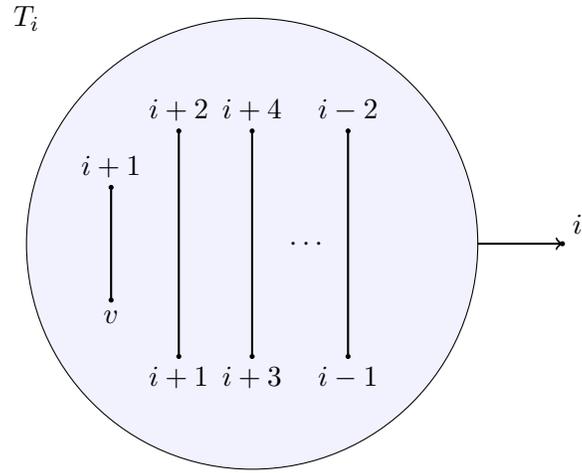
\begin{figure}
  \centering
      \begin{tikzpicture} [scale=0.75]
      \node at (-4,4) {$T_i$};
      
      \filldraw[color=black,fill=blue!5] (0,0) circle [radius=4];
      
      \filldraw[black] (-2.5,1) circle (1pt);
      \node [above] at (-2.5,1) {$i+1$};
      \filldraw[black] (-2.5,-1) circle (1pt);
      \node [below] at (-2.5,-1) {$v$};
      \draw[thick] (-2.5,1) -- (-2.5,-1);
      
      \filldraw[black] (-1.3,2) circle (1pt);
      \node [above] at (-1.3,2) {$i+2$};
      \filldraw[black] (0,2) circle (1pt);
      \node [above] at (0,2) {$i+4$};
      \filldraw[black] (1.7,2) circle (1pt);
      \node [above] at (1.7,2) {$i-2$};
      \filldraw[black] (-1.3,-2) circle (1pt);
      \node [below] at (-1.3,-2) {$i+1$};
      \filldraw[black] (0,-2) circle (1pt);
      \node [below] at (0,-2) {$i+3$};
      \filldraw[black] (1.7,-2) circle (1pt);
      \node [below] at (1.7,-2) {$i-1$};
      \draw[thick] (-1.3,2) -- (-1.3,-2);
      \draw[thick] (0,2) -- (0,-2);
      \draw[thick] (1.7,2) -- (1.7,-2);
      
      \node at (1,0) {$\cdots$};
      
      \filldraw[black] (5.5,0) circle (1pt);
      \node [above right] at (5.5,0) {$i$};
      \draw[thick,->] (4,0) -- (5.5,0);
      \end{tikzpicture}
  \caption{$T_i$ in the oriented extremal construction on $n+1$ vertices for even $n$}
\end{figure}

If $H$ does contain such a triple, then there exist three integers in $\mathbb{Z}_n$ that can be represented as both $\{a,a+2k,a+2k+1\}$ and $\{b,b+2i,b+2i+1\}$ with $a \neq b$. Without loss of generality we can assume that $b=0$. If $a+2k = 0$, then $a+2k+1 = 1$, but $1$ is not in any tail that points at $0$. Therefore, it must be the case that $a+ 2k+1 = 0$, but then $a+2k = n-1$. Therefore, the set is equal to $\{0,n-1,n-2\}$, and $a=n-2$, but $n-1$ does not point to $n-2$, a contradiction. Therefore, $H$ can have no such triple.

\begin{figure}
  \centering
       \begin{tikzpicture} [scale=0.75]
      \node at (-4,4) {$T_v$};
      
      \filldraw[color=black,fill=blue!5] (0,0) circle [radius=4];
      
      \filldraw[black] (-2.5,1) circle (1pt);
      \node [above] at (-2.5,1) {$0$};
      \filldraw[black] (-2.5,-1) circle (1pt);
      \node [below] at (-2.5,-1) {$\frac{n}{2}$};
      \draw[thick] (-2.5,1) -- (-2.5,-1);
      
      \filldraw[black] (-1.3,2) circle (1pt);
      \node [above] at (-1.3,2) {$1$};
      \filldraw[black] (0,2) circle (1pt);
      \node [above] at (0,2) {$2$};
      \filldraw[black] (1.7,2) circle (1pt);
      \node [above] at (1.7,2) {$\frac{n}{2} - 1$};
      \filldraw[black] (-1.3,-2) circle (1pt);
      \node [below] at (-1.3,-2) {$n-1$};
      \filldraw[black] (0,-2) circle (1pt);
      \node [below] at (0,-2) {$n-2$};
      \filldraw[black] (1.7,-2) circle (1pt);
      \node [below] at (1.7,-2) {$\frac{n}{2}+1$};
      \draw[thick] (-1.3,2) -- (-1.3,-2);
      \draw[thick] (0,2) -- (0,-2);
      \draw[thick] (1.7,2) -- (1.7,-2);
      
      \node at (1,0) {$\cdots$};
      
      \filldraw[black] (5.5,0) circle (1pt);
      \node [above right] at (5.5,0) {$v$};
      \draw[thick,->] (4,0) -- (5.5,0);
      \end{tikzpicture}
  \caption{$T_v$ in the oriented extremal construction on $n+1$ vertices for even $n$}
\end{figure}
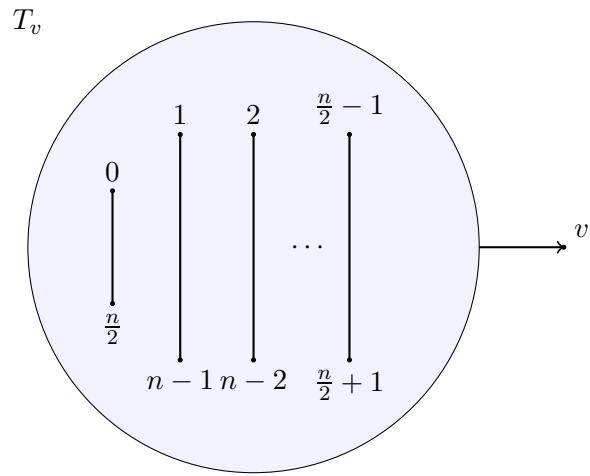

Now, we consider odd $n+1$. Here, let $V(H) = \mathbb{Z}_n \cup \{v\}$ where $v$ is a new vertex and use all the edges from the even construction plus some new ones that all contain $v$. So $E(H) = E_{even} \cup E_{new} \cup E_{v}$ where \[E_{even} = \bigcup_{i=0}^{n-1} \left\{(i+2k)(i+2k+1) \rightarrow i : k=1,2,\ldots,\frac{n-2}{2} \right\},\]and \[E_{new}= \left\{v(i+1) \rightarrow i: i=0,1,\ldots,n-1\right\}.\] Certainly, the construction has so far avoided the forbidden subgraph and given each of the first $n$ vertices the maximum number of tails. Now $E_v$ can be constructed as any set of $\frac{n}{2}$  disjoint pairs of elements from $\mathbb{Z}_n$ all pointing at $v$ so that no pair consists of two sequential numbers mod $n$. So any maximum matching of the $n$ elements that observes this condition will do.

In particular, we can let \[E_v = \left\{(i)(n-i) \rightarrow v : i = 1,\dots,\frac{n}{2}-1\right\} \cup \left\{(0)\left(\frac{n}{2}\right) \rightarrow v\right\}.\] So \[\text{ex}_o(n,I_1) = n \left\lfloor \frac{n-1}{2} \right\rfloor.\]
\end{proof}

\section{Forbidden $H_1$}

\begin{figure}
	\centering
	\begin{tikzpicture}
		\filldraw[black] (0,2) circle (1pt);
		\filldraw[black] (-1,0) circle (1pt);
		\filldraw[black] (1,0) circle (1pt);
		\filldraw[black] (1.5,1.5) circle (1pt);
		\filldraw[black] (-1.5,1.5) circle (1pt);
		
		\draw[thick] (0,2) -- (-1,0);
		\draw[thick] (0,2) -- (1,0);
		\draw[thick,->] (0.5,1) -- (1.5,1.5);
		\draw[thick,->] (-0.5,1) -- (-1.5,1.5);
		
		\node[above] at (0,2) {$x$};
		\node[above] at (-1.5,1.5) {$b$};
		\node[above] at (1.5,1.5) {$d$};
		\node[left] at (-1,0) {$a$};
		\node[right] at (1,0) {$c$};
		
	\end{tikzpicture}
	\caption{$H_1$}
	\label{B}
\end{figure}

In this section $H_1$ denotes the forbidden graph where two edges intersect in exactly one vertex such that it is in the tail for each edge. That is $V(H_1) = \{a,b,c,d,x\}$ and $E(H_1) = \{ax \rightarrow b, cx \rightarrow d\}$ (see Figure~\ref{B}). First we will show the following result for the oriented version of the problem.

\begin{theorem}
\label{exB}
For all $n \geq 6$, \[\text{ex}_o(n,H_1)=\left\lfloor \frac{n}{2} \right\rfloor (n-2).\]
\end{theorem}

We will use this result to solve the standard version of the problem and get the following.

\begin{theorem}
\label{TypeB}
For all $n \geq 8$, \[\text{ex}(n,H_1) = {n+1 \choose 2} - 3.\]
\end{theorem}

First, note that the proof of Theorem~\ref{exB} is straightforward when $n$ is even. To get a lower bound construction we can take a maximum matching of the $n$ vertices and use each pair of this matching as the tail set to point at all $n-2$ other vertices. That is, let $H$ be the graph with vertex set, \[V(H) = \{x_1,\ldots,x_{\frac{n}{2}}, y_1, \ldots, y_{\frac{n}{2}}\}\] and edge set, \[E(H) = \bigcup_{i=1}^{\frac{n}{2}} \left\{x_iy_i \rightarrow z: z \in V(H) \setminus \{x_i,y_i\}\right\}.\]

To show that this is also an upper bound, let $H$ be an $H_1$-free oriented graph on $n$ vertices. Then for any $x \in V(H)$, the directed link graph $D_x$ cannot have two independent edges (see Figure~\ref{Bdirlink}). Therefore, $D_x$ is either empty, a triangle, or a star with at most $n-2$ edges. Since $n \geq 5$, then $|D_x| \leq n-2$ for each $x$. So \[|E(H)| = \frac{1}{2} \sum_{x \in V(H)} |D_x| \leq \frac{1}{2}n(n-2).\] So we are finished for even $n$. However, this proof falls apart when $n$ is odd. We will need a different strategy.

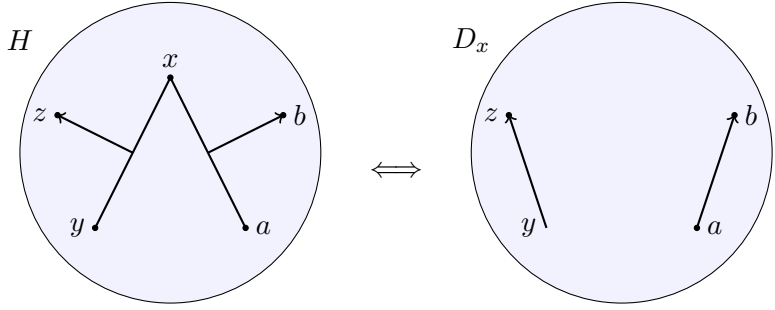
\begin{figure}
	\centering
	\begin{tikzpicture}
		\filldraw[color=black,fill=blue!5] (0,1) circle [radius=2];
	
		\node at (-2,2.5) {$H$};
		
		\filldraw[black] (0,2) circle (1pt);
		\node [above] at (0,2) {$x$};
		\filldraw[black] (-1,0) circle (1pt);
		\node [left] at (-1,0) {$y$};
		\filldraw[black] (1,0) circle (1pt);
		\node [right] at (1,0) {$a$};
		\filldraw[black] (1.5,1.5) circle (1pt);
		\node [right] at (1.5,1.5) {$b$};
		\filldraw[black] (-1.5,1.5) circle (1pt);
		\node [left] at (-1.5,1.5) {$z$};
		
		\draw[thick] (0,2) -- (-1,0);
		\draw[thick] (0,2) -- (1,0);
		\draw[thick,->] (0.5,1) -- (1.5,1.5);
		\draw[thick,->] (-0.5,1) -- (-1.5,1.5);
		
		\filldraw[color=black,fill=blue!5] (6,1) circle [radius=2];
		
		\node at (4,2.5) {$D_x$};
		
		\filldraw[black] (4.5,1.5) circle (1pt);
		\node [left] at (4.5,1.5) {$z$};
		\node [left] at (5,0) {$y$};
		\filldraw[black] (7,0) circle (1pt);
		\node [right] at (7,0) {$a$};
		\filldraw[black] (7.5,1.5) circle (1pt);
		\node [right] at (7.5,1.5) {$b$};
		
		\draw[thick,->] (7,0) -- (7.5,1.5);
		\draw[thick,->] (5,0) -- (4.5,1.5);
		
		\node at (3,0.75) {$\iff$};
	\end{tikzpicture}
	\caption{$H$ has a copy of $H_1$ with intersection vertex $x$ if and only if the directed link graph $D_x$ has a pair of disjoint directed edges.}
	\label{Bdirlink}
\end{figure}

\subsection{Counting edges by possible tail pairs}

The basis of our strategy in getting an upper bound on $\text{ex}_o(n,H_1)$ is to count the edges of an $H_1$-free graph $H$ by its tail sets. That is, \[|E(H)| = \sum_{\{x,y\} \in {V(H) \choose 2}} t(x,y)\] It is simple but important to note that if $H$ is $H_1$-free, then any two pairs of vertices that each point to two or more other vertices must necessarily be disjoint.

\begin{lemma}
Let $H$ be a $H_1$-free oriented graph. If $x_1,x_2,y_1,y_2 \in V(H)$ so that $t(x_1,y_1),t(x_2,y_2) \geq 2$ and $\{x_1,y_1\} \neq \{x_2,y_2\}$, then $\{x_1,y_1\} \cap \{x_2,y_2\} = \emptyset$.
\end{lemma}

\begin{proof}
Suppose, towards a contradiction, that $x_1 = x_2 = x$ but $y_1 \neq y_2$. Since $t(x,y_1) \geq 2$, then there exists some vertex $z_1$ distinct from $x$, $y_1$, and $y_2$ such that \[xy_1 \rightarrow z_1 \in E(H).\] Similarly, since $t(x,y_2) \geq 2$, then there exists some vertex $z_2$ distinct from $x$, $y_1$, and $y_2$ such that \[xy_2 \rightarrow z_2 \in E(H).\] If $z_1 \neq z_2$  this gives a copy of $H_1$.

So assume that they are the same vertex, $z_1=z_2=z$. Since $t(x,y_1) \geq 2$, then there is some second vertex that $x$ and $y_1$ point to that is distinct from $z$. The only choice that would not create a copy of $H_1$ with the edge $xy_2 \rightarrow z$ is $y_2$. Similarly, since $t(x,y_2) \geq 2$, then there is some second vertex that $x$ and $y_2$ point to that is distinct from $z$. The only choice that would not create a copy of $H_1$ with the edge $xy_1 \rightarrow z$ is $y_1$. So \[xy_1 \rightarrow y_2, xy_2 \rightarrow y_1 \in E(H)\] which contradicts the fact that $H$ is oriented.
\end{proof}

Therefore, if we assume that $H$ is $H_1$-free on $n$ vertices, then we can split its vertices up into $k$ disjoint pairs that each serve as tail sets to at least two edges of $H$ plus a set of $n-2k$ vertices that belong to no such pair. That is, \[V(H) = \{x_1,y_1,\ldots,x_k,y_k\} \cup R\] so that $t(x_i,y_i) \geq 2$ for $i=1,\ldots,k$ and $t(w,v) \leq 1$ for all other vertex pairs, $\{w,v\}$ (see Figure~\ref{twopluspairs}).

\begin{figure}
	\centering
	\begin{tikzpicture}
	\node at (-1,4) {$H$};
	
	\filldraw[black] (0,3) circle (1pt);
	\node [left] at (0,3) {$x_1$};
	\filldraw[black] (0,2) circle (1pt);
	\node [left] at (0,2) {$x_2$};
	\filldraw[black] (0,0) circle (1pt);
	\node [left] at (0,0) {$x_k$};
	
	\filldraw[black] (2,3) circle (1pt);
	\node [right] at (2,3) {$y_1$};
	\filldraw[black] (2,2) circle (1pt);
	\node [right] at (2,2) {$y_2$};
	\filldraw[black] (2,0) circle (1pt);
	\node [right] at (2,0) {$y_k$};
	
	\draw[thick] (0,3) -- (2,3);
	\draw[thick] (0,2) -- (2,2);
	\draw[thick] (0,0) -- (2,0);
	
	\node at (1,1) {$\vdots$};
	
	\filldraw[color=black,fill=blue!5] (5,1.5) circle [radius=1.5];
	\node at (5,1.5) {$R$};
	\end{tikzpicture}
	\caption{An $H_1$-free graph on $n$ vertices breaks down into $k$ disjoint pairs that each point to at least two other vertices plus a remainder set $R$ with $n-2k$ vertices that belong to no such pair.}
	\label{twopluspairs}
\end{figure}
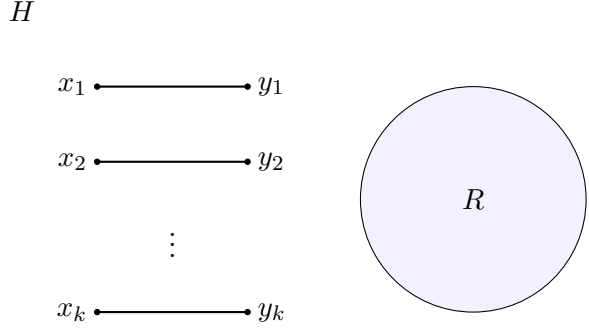

We now have two cases to consider. Either there are no such pairs ($k=0$) or there is at least one ($k \geq 1$).

\subsection{No pair points to more than one vertex ($k=0$)}

Assume that $k=0$. Then $t(x,y) \leq 1$ for every pair $\{x,y\} \in {V(H) \choose 2}$. If $|D_x| \leq n-3$ for all $x \in V(H)$, then \[|E(H)| = \frac{1}{2} \sum_{x \in V(H)} |D_x| \leq \frac{1}{2}n(n-3) < \frac{1}{2}(n-1)(n-2)\] and we are done. Otherwise, there exists some vertex $x$ that belongs to $n-2$ tail sets. Therefore, $D_x$ is a star of directed edges with some focus $y$. Either $t(x,y)=0$ or $t(x,y)=1$.

If $t(x,y) = 0$, then all of the $n-2$ directed edges of $D_x$ must point to $y$ (see Figure~\ref{CaseApic}). Such a configuration in $H$ limits the number of edges to ${n-1 \choose 2}$ as proven in Lemma~\ref{CaseA}.

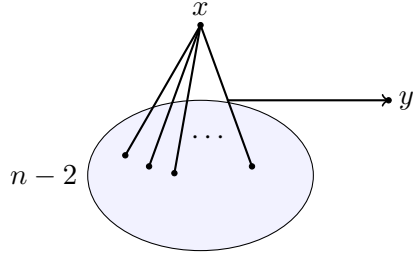
\begin{figure}
	\centering
	\begin{tikzpicture}
		\filldraw[color=black,fill=blue!5] (0,-2) ellipse (1.5 and 1);
		\node [left] at (-1.5, -2) {$n-2$};
	
		\filldraw[black] (0,0) circle (1pt);
		\filldraw[black] (xyz polar cs:angle=240, radius=2) circle (1pt);
		\filldraw[black] (xyz polar cs:angle=250, radius=2) circle (1pt);
		\filldraw[black] (xyz polar cs:angle=260, radius=2) circle (1pt);
		\filldraw[black] (xyz polar cs:angle=290, radius=2) circle (1pt);
		\filldraw[black] (2.5,-1) circle (1pt);
		
		\draw[thick] (0,0)--(xyz polar cs:angle=240, radius=2);
		\draw[thick] (0,0)--(xyz polar cs:angle=250, radius=2);
		\draw[thick] (0,0)--(xyz polar cs:angle=260, radius=2);
		\draw[thick] (0,0)--(xyz polar cs:angle=290, radius=2);
		\draw[thick,->] (0.35,-1)--(2.5,-1);
		
		\node [above] at (0,0) {$x$};
		\node at (xyz polar cs:angle=275, radius=1.5) {$\cdots$};
		\node [right] at (2.5,-1) {$y$};
	\end{tikzpicture}
	\caption{The special case configuration discussed in Lemma~\ref{CaseA}. Here, vertex $x$ joins with every other element to point to vertex $y$.}
	\label{CaseApic}
\end{figure}

On the other hand, if $t(x,y) = 1$, then $xy \rightarrow z \in E(H)$ for some vertex $z$, and $xv \rightarrow y$ for all other vertices $v \neq x,y,z$. Such a configuration in $H$ will limit the number of edges to ${n-1 \choose 2}$ as proven in Lemma~\ref{CaseB}.

\begin{lemma}
\label{CaseA}
Let $H$ be an oriented graph on $n \geq 6$ vertices such that $t(x,y) \leq 1$ for each pair $\{x,y\} \in {V(H) \choose 2}$. If $H$ is $H_1$-free and contains vertices $x$ and $y$ such that $xv \rightarrow y \in E(H)$ for each $v \in V(H) \setminus \{x,y\}$, then \[|E(H)| \leq {n-1 \choose 2}.\] See Figure~\ref{CaseApic}.
\end{lemma}

\begin{proof}
We want to show that there can be no more than ${n-2 \choose 2}$ additional edges in $H$ other than the $n-2$ edges described in the statement of the lemma. This would give an upper bound on the total number of edges of \[{n-2 \choose 2} + (n-2) = {n-1 \choose 2}.\]

First, note that every triple of the form $\{x,y,v\}$ already holds an edge. This implies that any additional edge cannot contain both $x$ and $y$ since $H$ is oriented. On the other hand, if we were to add an edge, $vw \rightarrow u$, that excluded both $x$ and $y$ completely, then this new edge would create a copy of $H_1$ with the existing edge, $vx \rightarrow y$. Therefore, every additional edge must be on a triple of the form $\{v,w,x\}$ or $\{v,w,y\}$.

However, $x$ is already in the maximum number of tails. So given any pair of non-$\{x,y\}$ vertices, $\{v,w\}$, the only possible additional edges are \[vw \rightarrow x, vw \rightarrow y, yv \rightarrow w, \text{and } yw \rightarrow v.\] The last three all appear on the triple, $\{v,w,y\}$ and are therefore mutually exclusive choices when it comes to adding them to the graph. The first two are also mutually exclusive choices since $t(v,w) \leq 1$.

So assume, towards a contradiction, that we could add ${n-2 \choose 2} + 1$ more edges to the existing configuration. Then some pair $\{v,w\}$ of non-$\{x,y\}$ vertices must be used twice. Without loss of generality, this means we must add the edges $vw \rightarrow x$ and $yv \rightarrow w$.

Now, let $u$ be any of the remaining $n-4$ vertices. The possible edge $uv \rightarrow y$ would create a copy of $H_1$ with $vw \rightarrow x$, and the possible edge $uv \rightarrow x$ would create a copy of $H_1$ with $vy \rightarrow w$. Therefore, the pair $\{v,u\}$ cannot be a tail set for any edge.

We can also view the potential additional edges as two different types: those that have a tail set of two non-$\{x,y\}$ vertices and those that have a tail set that includes $y$. There were originally at most ${n-2 \choose 2}$ of the first type that we are allowed to add in total, one edge for every distinct pair. However, $v$ can now no longer be in a tail set with any of the other $n-4$ vertices. So there are now at most ${n-2 \choose 2} - (n-4)$ edges of this first type left possible to add. Therefore, in order to add ${n-2 \choose 2} + 1$ edges over all, we will need at least $n-3$ of them to be of the second type - those that have $y$ in the tail set.

Note that $x$ must be an isolated vertex in the directed link graph $D_y$. Hence, there are at most $n-3$ tails containing $y$ since otherwise the directed graph $D_y$ would have $n-2$ edges among $n-2$ vertices. In this case, $D_y$ would have two independent directed edges and so $H$ would have a copy of $H_1$ with $y$ as its intersection vertex. Moreover, $D_y$ must be a star with a single vertex of intersection. Since $v \rightarrow w \in E(D_y)$, then this vertex of intersection must either be $v$ or $w$.

Hence, in order to add ${n-2 \choose 2} + 1$ edges, we'll need to have ${n-2 \choose 2} - (n-4)$ edges that have non-$\{x,y\}$ tail sets. Since the tail set, $\{v,w\}$, already points to $x$, then this implies that all such edges must also point to $x$. Otherwise, we'd have some edge of the form $ab \rightarrow y$. If $a=w$ or $b=w$, then this would create a copy of $H_1$ with $vw \rightarrow x$. If both elements are distinct from $w$, then we would still need to point the pair $wa$ either to $x$ or to $y$. Either choice would create a copy of $H_1$.

Let $u$ be one of the remaining vertices. Then $u$ must be adjacent to a directed edge of $D_y$ for there to be $n-3$ edges added with $y$ in the tail set. If $v$ is the vertex of intersection of $D_y$, then this edge must either be $u \rightarrow v$ or $v \rightarrow u$. Either yields a copy of $H_1$. Similarly, if $w$ is the vertex of intersection of $D_y$, then either $wy \rightarrow u \in E(H)$ or $uy \rightarrow w \in E(H)$. Again, either of these yields a copy of $H_1$. Therefore, ${n-2 \choose 2} + 1$ edges cannot be added to the existing configuration.
\end{proof}

\begin{lemma}
\label{CaseB}
Let $H$ be an oriented graph on $n \geq 6$ vertices such that for each pair $x,y \in V(H)$, $t(x,y) \leq 1$. If $H$ is $H_1$-free and contains vertices $x$, $y$, and $z$ such that $xy \rightarrow z \in E(H)$ and $xv \rightarrow y \in E(H)$ for each $v \in V(H) \setminus \{x,y,z\}$ (see Figure~\ref{Bpic}), then \[|E(H)| \leq {n-1 \choose 2}.\]
\end{lemma}

\begin{figure}
	\centering
	\begin{tikzpicture}
		\filldraw[color=black,fill=blue!5] (0,-2) ellipse (1.5 and 1);
		\node [left] at (-1.5, -2) {$n-3$};
	
		\filldraw[black] (0,0) circle (1pt);
		\filldraw[black] (xyz polar cs:angle=240, radius=2) circle (1pt);
		\filldraw[black] (xyz polar cs:angle=250, radius=2) circle (1pt);
		\filldraw[black] (xyz polar cs:angle=260, radius=2) circle (1pt);
		\filldraw[black] (xyz polar cs:angle=290, radius=2) circle (1pt);
		\filldraw[black] (2.5,-1) circle (1pt);
		\filldraw[black] (1.75,0.75) circle (1pt);
		
		\draw[thick] (0,0)--(xyz polar cs:angle=240, radius=2);
		\draw[thick] (0,0)--(xyz polar cs:angle=250, radius=2);
		\draw[thick] (0,0)--(xyz polar cs:angle=260, radius=2);
		\draw[thick] (0,0)--(xyz polar cs:angle=290, radius=2);
		\draw[thick,->] (0.35,-1)--(2.5,-1);
		\draw[thick] (0,0)--(2.5,-1);
		\draw[thick,->] (1.25,-0.5)--(1.75,0.75);
		
		\node [above] at (0,0) {$x$};
		\node at (xyz polar cs:angle=275, radius=1.5) {$\cdots$};
		\node [right] at (2.5,-1) {$y$};
		\node [right] at (1.75,0.75) {$z$};
	\end{tikzpicture}
	\caption{The special case configuration discussed in Lemma~\ref{CaseB}. Here, $x$ joins with every vertex except $z$ to point to $y$ and then joins with $y$ to point to $z$.}
	\label{Bpic}
\end{figure}
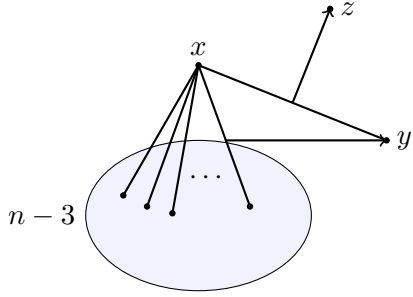

\begin{proof}
Let $W = \{1,2,\ldots,n-3\}$ be the set of non-$\{x,y,z\}$ vertices. Any additional edge to this graph must have a tail set of the form $\{i,j\}$, $\{i,y\}$, $\{i,z\}$, or $\{y,z\}$ for $i,j \in W$. An $ij$ tail can only point to $x$ or to $y$ and there are ${n-3 \choose 2}$ pairs like this possible. An $iy$ tail cannot point to $x$ because $H$ is oriented. It cannot point to $j$ since that would create a copy of $H_1$ with $xy \rightarrow z$. Therefore, it could only point to $z$. An $iz$ tail could not point to any $j$ since this would create a copy of $H_1$ with the edge $ix \rightarrow y$. Therefore, it could only point to $y$ or to $x$. And a $yz$ tail could not point to $x$ since $H$ is oriented. Therefore, it could only point to some $i$.

Assume, towards a contradiction, that we can add \[{n-1 \choose 2} + 1= {n-3 \choose 2} + (n-3)+ 1\] edges to the existing configuration. Since we can add at most ${n-3 \choose 2}$ edges with tail sets made entirely of vertices from $W$, then we must have at least $n-2$ additional edges from the other possibilities.

For each $i \in W$ we could have \[iy \rightarrow z, yz \rightarrow i, iz \rightarrow y, \text{and } iz \rightarrow x.\] The first three of these are mutually exclusive choices since they are all on the same triple. Similarly, the last two are mutually exclusive choices since we are only allowing up to one edge per possible tail set.

Therefore, in order to add $n-2$ of these types of edges, two must use the same element of $W$. Given the mutually exclusive choices above this implies that there is some vertex $i \in W$ such that either $iz \rightarrow x, yi \rightarrow z \in E(H)$ or $iz \rightarrow x, yz \rightarrow i \in E(H)$.

In the first case, $ij$ is no longer a possible tail for any edge for all $n-4$ remaining vertices $j \in W$. This is because $iz \rightarrow x$, $yi \rightarrow z$, and $ix \rightarrow y$ create a triangle in $D_i$. So any additional edge with $i$ in the tail would give two independent edges in $D_i$ and therefore a copy of $H_1$.

Hence, we can get at most ${n-3 \choose 2} - (n-4)$ edges with tails in $W$. This means that we will need $2(n-3)$ edges from the other possible edges to make up the difference if we want to add \[{n-3 \choose 2} + (n-3) + 1\] more edges.

Since each of the $n-3$ vertices from $W$ can be in up to two of these additional edges, then $iz \rightarrow x$ would need to be an edge for every $i \in W$ and that $\{y,z,i\}$ also needs to hold one edge for every $i \in W$.

If $yz \rightarrow i$ is used once, then we get a copy of $H_1$ with $jz \rightarrow x$ for some other $j \in W$. Therefore, for all $i \in W$ we must have the edges $iy \rightarrow z$ and $iz \rightarrow x$. However, any pair $i,j \in W$ can now point to nothing since the only possibilities for such a tail were $x$ or $y$ to begin with and both of these options create copies of $H_1$. So in this case the most that we can add is \[2(n-3) \leq {n-3 \choose 2} + (n-3)\] for all $n \geq 6$.

In the other case we have added $iz \rightarrow x$ and $yz \rightarrow i$ for some $i$. Which means that $yz \rightarrow j$ is not allowed for any $j \neq i$ from $W$. Also, $jz \rightarrow y$ would make a copy of $H_1$ with $iz \rightarrow x$ and $jz \rightarrow x$ would make a copy of $H_1$ with $yz \rightarrow i$. Therefore, for all $j \neq i$ we can only add the edge $jy \rightarrow z$.

In order to add ${n-3 \choose 2} + n-2$ edges, we'll need all of these as well as all possible edges with tails in $W$. However, since $iz \rightarrow x$, all of these edges with tails completely in $W$ must also point to $x$. Otherwise, some pair $ab$ would point to $y$. If $a=i$ or $b=i$, then this would make a copy of $H_1$ with $iz \rightarrow x$. If $i \neq a,b$, then consider where the pair $ai$ points. It must either point to $x$ or to $y$, but either of these would create a copy of $H_1$.

So all pairs of $W$ must point to $x$ and for all $j \in W$ not equal to $i$ we must have the edge $jy \rightarrow z$. But $jy \rightarrow z$ and $ij \rightarrow x$ create a copy of $H_1$, a contradiction. Hence, it is not possible to add more than ${n-3 \choose 2} + (n-3)$ edges to the configuration. Since the configuration already has $n-2$ edges, then there can be no more than ${n-1 \choose 2}$ edges total.
\end{proof}

Together these two lemmas take care of the cases where all pairs of vertices point to at most one vertex in $H$.

\subsection{At least one pair of vertices is the tail set to more than one edge of $H$ ($k > 0$)}

So let's return to our description of an $H_1$-free oriented graph as being made up of $k \geq 1$ vertex pairs that each serve as tails to strictly more than one edge plus a set $R$ of the remaining $n-2k$ vertices, \[V(H) = \{x_1,y_1,\ldots,x_k,y_k\} \cup R\] (see Figure~\ref{twopluspairs}). For each pair $\{x_i,y_i\}$ we want to prove the following upper bound, \[t(x_i,y_i) + \sum_{v \neq x_i,y_i} \left(t(x_i,v)+t(y_i,v)\right) \leq n-2.\] That is, the total number of edges that include either $x_i$ or $y_i$ or both in the tail set is at most $n-2$.

Now, \[|E(H)| = \sum_{\{x,y\} \in {V(H) \choose 2}} t(x,y) \leq \sum_{i=1}^k \left(t(x_i,y_i) + \sum_{v \neq x_i,y_i} \left(t(x_i,v)+t(y_i,v)\right) \right) + \sum_{\{x,y\} \in {R \choose 2}} t(x,y).\] Note that each pair of vertices in $R$ act as a tail set at most once so \[\sum_{\{x,y\} \in {R \choose 2}} t(x,y) \leq {n-2k \choose 2}.\] Therefore, proving the upper bound for each $\{x_i,y_i\}$ pair would imply that \[|E(H)| \leq k(n-2) + {n-2k \choose 2}.\] Since \[k(n-2) + {n-2k \choose 2} = 2k^2 - (n+1)k + {n \choose 2}\] is a quadratic polynomial with positive leading coefficient in terms of $k$, then it is maximized at the endpoints. Here, that means at $k=1$ and at $k= \left\lfloor\frac{n}{2}\right\rfloor$.

When $n$ is odd, both of these values for $k$ give the upper bound, \[|E(H)| \leq {n-1 \choose 2}.\] Only when $n$ is even can we do better and get \[|E(H)| \leq \frac{n(n-2)}{2}\] in the case where $k=\frac{n}{2}$. In either case this give an upper bound of \[|E(H)| \leq \left\lfloor \frac{n}{2} \right\rfloor (n-2).\]

So we need only prove that, in general,  \[t(x_i,y_i) + \sum_{v \neq x_i,y_i} \left(t(x_i,v)+t(y_i,v)\right) \leq n-2.\] This is straightforward to show if $t(x_i,y_i) \geq 3$. However, when $t(x_i,y_i)=2$ there is a case where it fails to hold. This is taken care of in the following lemma.

\begin{lemma}
\label{CaseC}
Let $H$ be an oriented graph on $n \geq 6$ vertices. If $H$ is $H_1$-free and contains vertices $x$, $y$, $a$, and $b$ such that $\{x,y\}$ is the tail set to exactly 2 edges with \[xy \rightarrow a, xy \rightarrow b, yb \rightarrow a  \in E(H),\] and for each $v \in V(H) \setminus \{x,y,a,b\}$, $xv \rightarrow y$ (see Figure~\ref{CaseCpic}), then \[|E(H)| \leq {n-1 \choose 2}.\]
\end{lemma}

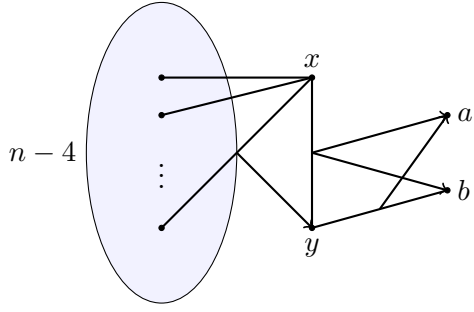
\begin{figure}
	\centering
	\begin{tikzpicture}
	\filldraw[color=black,fill=blue!5] (-2,1) ellipse (1 and 2);
	\filldraw[black] (-2,2) circle (1pt);
	\filldraw[black] (-2,1.5) circle (1pt);
	\filldraw[black] (-2,0) circle (1pt);
	
	\node [left] at (-3,1) {$n-4$};
	\node at (-2,0.8) {$\vdots$};
	
	\draw[thick] (0,2) -- (-2,2);
	\draw[thick] (0,2) -- (-2,1.5);
	\draw[thick] (0,2) -- (-2,0);
	\draw[thick,->] (-1,1) -- (0,0);
	
	\filldraw[black] (0,2) circle (1pt);
	\filldraw[black] (0,0) circle (1pt);
	\filldraw[black] (1.8,1.5) circle (1pt);
	\filldraw[black] (1.8,0.5) circle (1pt);
	
	\node [above] at (0,2) {$x$};
	\node [below] at (0,0) {$y$};
	\node [right] at (1.8,1.5) {$a$};
	\node [right] at (1.8,0.5) {$b$};
	
	\draw[thick] (0,2) -- (0,0);
	\draw[thick, ->] (0,1) -- (1.8,1.5);
	\draw[thick, ->] (0,1) -- (1.8,0.5);
	\draw[thick] (0,0) -- (1.8,0.5);
	\draw[thick, ->] (0.9,0.25) -- (1.8,1.5);
	\end{tikzpicture}
	\caption{An $H_1$-free graph containing this configuration with have at most ${n-1 \choose 2}$ edges as shown in Lemma~\ref{CaseC}.}
	\label{CaseCpic}
\end{figure}

\begin{proof}
First consider which pairs of vertices could possibly be a tail set to an edge in this graph. Let $W=\{1,\ldots,n-4\}$ be the set of vertices other than $\{x,y,a,b\}$. Then $\{i,j\}$ can be a tail set to $ij \rightarrow x$ and $ij \rightarrow y$ for any pair $i,j \in W$. Since $xy \rightarrow a$, then $xi$ can point to nothing other than $y$. Similarly, $xa$ and $xb$ could only possibly point to $b$ and $a$ respectively, but either would create a copy of $H_1$ with $xi \rightarrow y$ for any $i \in W$. Also, by assumption $xy$ can point to nothing else. Hence, $x$ is in no additional tail sets.

Since $yb \rightarrow a$ and $xy \rightarrow a$, then $ya$ cannot point to $b$ or to $x$. It can also not point to any $i \in W$ since this would create a copy of $H_1$ with $xy \rightarrow b$. So $y$ can be in no additional tails. The pair $ab$ can point to anything aside from $y$ since $H$ is oriented, and $ai$ can point to $x$ or $y$ for any $i \in W$ but not to $b$ or another element of $W$ since either would create a copy of $H_1$ with $xi \rightarrow y$. Similarly, $bi$ can point to $y$ for each $i \in W$ but not to $x$ or to $a$ or to another element of $W$ since these would create a copy of $H_1$ with either $yb \rightarrow a$ or $xi \rightarrow y$.

Leaving aside the edges with tail sets completely in $W$ for the moment, this means there are $4(n-4) + 1$ possible edges remaining. There are $n-4$ each of types $ai \rightarrow x$, $ai \rightarrow y$, $bi \rightarrow y$, and $ab \rightarrow i$ plus one extra edge which is $ab \rightarrow x$.

Suppose we are able to use at least $2(n-4)+1$ of these edges. First, if one of them is $ab \rightarrow x$, then there could be none of the types $ai \rightarrow y$ or $bi \rightarrow y$. So all of the ones of type $ab \rightarrow i$ and $ai \rightarrow x$ would need to be used. But since $n \geq 6$ there are at least two vertices in $W$. So there would exist edges $ai \rightarrow x$ and $ab \rightarrow j$ with $i \neq j$, a copy of $H_1$. Therefore $ab \rightarrow x$ cannot be used if we want to get more than $2(n-4)$ of these edges.

Hence, we need to use at least three types of edges from the four possible types. Since any of the types $ai \rightarrow x$, $ai \rightarrow y$, and $bi \rightarrow y$ eliminate the possibility of using any edge $ab \rightarrow j$ where $j \neq i$, then we can use at most one of this last type of edge. But since $n \geq 6$, then $2(n-4)+1 \geq 5$ which means one of the other types gets used at least twice. Regardless of which one it is, there can be nothing used from the $ab \rightarrow i$ types of edges.

Therefore, we must use $2(n-4)+1$ edges from only the first three types. So there must be a vertex from $W$ that belongs to three of these edges, say \[ai \rightarrow y, bi \rightarrow y, \text{and } ai \rightarrow x.\] Then for any $j \neq i$, $aj \rightarrow x$ creates a copy of $H_1$ with $ai \rightarrow y$ and $aj \rightarrow x$ creates a copy of $H_1$ with $ai \rightarrow x$. So at most $2 + (n-4) < 2(n-4)$ edges could be used. Thus, at most $2(n-4)$ of these kinds of edges can be used over all.

Now let us look at the edges with tail sets contained in $W$. We have seen that each $ij$ can point to $x$ or to $y$, but nothing so far has kept the pair from pointing to both. However, if some pair does point to both, then no other tail could use either of these vertices since this would create a copy of $H_1$. Therefore, if there are $l$ such pairs, then there are at most $2l + {n - 4 - 2l \choose 2}$ edges with tails from $W$. But since $n \geq 5$, then $\frac{n+5}{2} \leq n$. And since $l \leq \frac{n-4}{2}$, then \[2l + {n-4-2l \choose 2} \leq {n-4 \choose 2}.\] So there are at most ${n-1 \choose 2}$ edges in $H$.
\end{proof}

\subsection{First main result, $\text{ex}_o(n,H_1) = \left\lfloor \frac{n}{2} \right\rfloor (n-2)$.}

Now we can proceed with establishing the upper bound under the assumption that the configuration presented in Lemma~\ref{CaseC} does not occur in our directed hypergraph. As we've seen, all that's necessary to show is that \[t(x_i,y_i) + \sum_{v \neq x_i,y_i} \left(t(x_i,v)+t(y_i,v)\right) \leq n-2\] for any pair of vertices $\{x_i,y_i\}$ that serves as the tail set to at least two edges.

So let $\{x,y\}$ be such a pair, and divide the rest of the vertices of $H$ into two groups, those that are a head vertex to some edge with $xy$ as the tail and those that are not. That is, \[V(H) \setminus \{x,y\} = \{h_1,\ldots,h_m\} \cup \{n_1,\ldots,n_t\}\] where for each $i=1,\ldots,m$, there exists an edge, $xy \rightarrow h_i \in E(H)$ and for each $j=1,\ldots,t$, $xy \rightarrow n_j \not \in E(H)$ (note that $t(x,y)=m$ and that $m+t=n-2$).

Now, consider an edge that contains either $x$ or $y$ in the tail but not both. Then the other tail vertex is either some $h_i$ or some $n_j$. In the case of $n_j$, this edge must either be of the form $xn_j \rightarrow y$ or $yn_j \rightarrow x$ to avoid a copy of $H_1$ with both $xy \rightarrow h_1$ and $xy \rightarrow h_2$. Moreover, since $H$ is oriented, there can be at most one. Hence, \[\sum_{j=1}^t \left(t(x,n_j)+t(y,n_j)\right) \leq t.\]

Now consider a tail set that includes either $x$ or $y$ and some $h_i$. Without loss of generality, assume that $xh_1$ is the tail to some edge. Since $t(x,y) \geq 2$, then there is some other vertex $h_2$ such that $xy \rightarrow h_2 \in E(H)$. In order to avoid a copy of $H_1$ with this edge, $xh_1$ must either point to $y$ or to $h_2$. However, $xh_1 \rightarrow y \not \in E(H)$ since this would give the triple $\{x,y,h_1\}$ more than one edge.

Therefore, $xh_1 \rightarrow h_2$ is the only option. However, if $t(x,y) \geq 3$, then this will create a copy of $H_1$ with $xy \rightarrow h_3$. So $xh_i$ and $yh_i$ cannot be tails to any edge. So \[\sum_{i=1}^m \left(t(x,h_i)+t(y,h_i)\right) = 0.\] Therefore,
\begin{align*}
&t(x,y) + \sum_{v \neq x,y} \left(t(x,v)+t(y,v)\right)\\
&= m + \sum_{j=1}^t \left(t(x,n_j)+t(y,n_j)\right) + \sum_{i=1}^m \left(t(x,h_i)+t(y,h_i)\right)\\
&\leq m+t\\
&=n-2
\end{align*}
when $t(x,y) \geq 3$.

The only other possibility is that $t(x,y)=2$. So suppose this is the case and that the head vertices to $xy$ are $a$ and $b$. Without loss of generality, assume that $yb \rightarrow a \in E(H)$. Note that this precludes any edges of the form $yn_j \rightarrow x$. Similarly, if we added the edge $xa \rightarrow b$ or the edge $xb \rightarrow a$, then we could not add any edges of the form $xn_j \rightarrow y$ and so \[\sum_{j=1}^t \left(t(x,n_j)+t(y,n_j)\right) = 0.\] Moreover, $ya \rightarrow b$ would lead to more than one edge on the triple $\{y,a,b\}$. So \[\sum_{i=1}^m \left(t(x,h_i)+t(y,h_i)\right) = 2\] and total we would have, \[t(x,y) + \sum_{v \neq x,y} \left(t(x,v)+t(y,v)\right) = 4 \leq n-2.\]

On the other hand, if $xa$ and $xb$ are not tails to any edge, then the only way we could get a sum more than $n-2$ is if $xn_j \rightarrow y \in E(H)$ for all $j=1,\ldots, n-4$. But this is exactly the configuration described in Lemma~\ref{CaseC} which we have excluded.

Therefore, \[t(x,y) + \sum_{v \neq x,y} \left(t(x,v)+t(y,v)\right) \leq n-2\] for any such pair, and this is enough to establish that \[\text{ex}_o(n,H_1) \leq \left\lfloor \frac{n}{2} \right\rfloor (n-2).\]

Conversely, we have already considered an extremal construction in the case where $n$ is even, and this same construction will work when $n$ is odd. That is, take a maximum matching of the vertices (leaving one out) and then use each matched pair as the tail set for all $n-2$ possible edges.

Another construction that works for odd $n$ that is not extremal for even $n$ is to designate one vertex as the only head vertex and then make all ${n-1 \choose 2}$ pairs of the rest of the vertices tail sets.

Therefore, \[\text{ex}_o(n,H_1) = \left\lfloor \frac{n}{2} \right\rfloor (n-2).\]

Also, note that the only way that any construction could have more than ${n-1 \choose 2}$ edges is if $n$ is even \emph{and} the vertices are partitioned into $\frac{n}{2}$ pairs such that each points to at least two other vertices. This fact comes directly from the requirement that $k=\frac{n}{2}$ in the optimization of \[k(n-2) + {n-2k \choose 2}\] in order for the expression to be more than ${n-1 \choose 2}$.

\subsection{Intersections of multiedge triples in the standard version}

Now, let $H$ be an $H_1$-free graph on $n$ vertices under the standard version of the problem so that any triple of vertices can now have up to all three possible directed edges. If we let $t_H$ be the number of triples of vertices of $H$ that hold at least one edge, and we let $m_H$ be the number of triples that hold at least two, then we have the following simple observation: \[|E(H)| \leq t_H + 2m_H.\]

We start our path towards an upper bound on $|E(H)|$ by finding an upper bound on the number of multiedge triples, $m_H$. We will need to prove some facts about the multiedge triples of $H$. First, any triple which holds two edges of $H$ might as well hold three.

\begin{lemma}
\label{more}
Let $H$ be an $H_1$-free graph such that some triple of vertices $\{x,y,z\}$ contains two edges. Define $H'$ by $V(H') = V(H)$ and \[E(H') = E(H) \cup \{xy \rightarrow z,xz \rightarrow y,yz \rightarrow x\}.\]Then $H'$ is also $H_1$-free.
\end{lemma}

\begin{proof}
Suppose $H'$ is not $H_1$-free. Since $H$ is $H_1$-free and the two graphs differ by at most one edge, then they must differ by exactly one edge. Without loss of generality, say \[\{xy \rightarrow z\} = E(H') \setminus E(H).\] This edge must be responsible for creating the copy of $H_1$ in $H'$. So it must intersect another edge in exactly one vertex that is in the tail set of both.

Therefore, without loss of generality, there is an edge $xt \rightarrow s \in H$ where $\{s,t\} \cap \{y,z\} = \emptyset$. However, since $\{x,y,z\}$ already contained two edges of $H$, then $xz \rightarrow y \in H$. Since $xt \rightarrow s$ and $xz \rightarrow y$ make a copy of $H_1$, then $H$ cannot be $H_1$-free, a contradiction.
\end{proof}

Next, we want to show that no two multiedge triples can intersect in exactly one vertex.

\begin{lemma}
\label{multilimit}
Let $H$ be a $H_1$-free graph. If two vertex triples $\{x,y,z\}$ and $\{s,t,r\}$ each contain two or more edges of $H$, then \[|\{x,y,z\} \cap \{s,t,r\}| \neq 1.\]
\end{lemma}

\begin{proof}
Suppose \[|\{x,y,z\} \cap \{s,t,r\}| = 1\]By Lemma~\ref{more}, since $H$ is $H_1$-free, the graph created from $H$ by adding all three possible edges on the triples $\{x,y,z\}$ and $\{s,t,r\}$ is also $H_1$-free. But if $x=r$ and $x$, $y$, $z$, $s$, and $t$ are all distinct, then this graph contains $xy \rightarrow z$ and $xs \rightarrow t$ which is a copy of $H_1$, a contradiction.
\end{proof}

Therefore, we can use an upper bound on the number of undirected 3-uniform hyperedges such that no two intersect in exactly one vertex as an upper bound on the number of multiedge triples. Moreover, the extremal examples are easy to describe which will be important for finding the upper bound for $\text{ex}(n,H_1)$ as well as for establishing the uniqueness of the lower bound construction.

\begin{lemma}
\label{undirected}
Let $H$ be a 3-uniform undirected hypergraph on $n$ vertices such that no two edges intersect in exactly one vertex, then \[|E(H)| \leq 
\begin{cases}
n & n \equiv 0 \text{ mod } 4\\
n-1 & n \equiv 1 \text{ mod } 4\\
n-2 & n \equiv 2,3 \text{ mod } 4\\
\end{cases}\]and $H$ is the disjoint union of $K_4^{(3)}$s, $K_4^{(3)}$s minus an edge ($K_4^{-}$), and sunflowers with a common intersection of two vertices.
\end{lemma}

\begin{proof}
Two edges of $H$ are either disjoint or they intersect in two vertices. So connected components of $H$ that have 1 or 2 edges are both sunflowers. A third edge can be added to a two-edge sunflower by either using the two common vertices to overlap with both edges in two or by using one common vertex and the two petal vertices. So a connected component of $H$ with 3 edges is either a sunflower or a $K_4^{-}$.

The only way to connect a fourth edge to the three-edge sunflower is to make a four-edge sunflower, and this is true for a $k$-edge sunflower to a $(k+1)$-edge sunflower for all $k \geq 3$. The only way to add a fourth edge to the $K_4^{-}$ is to make a $K_4^{(3)}$ and then no new edges may be connected to a $K_4^{(3)}$ without intersecting two of its edges in exactly one vertex each. Therefore, these are the only possible connected components of $H$.

A sunflower with $k$ edges uses $k+2$ vertices, and a $K_4^{(3)}$ has four edges on 4 vertices. Therefore, if $n \equiv 0 \text{ mod } 4$ we can get at most $n$ edges with a disjoint collection of $K_4^{(3)}$s. Similarly, the best we can do when $n \equiv 1 \text{ mod } 4$ is $n-1$ edges with a disjoint collection of $K_4^{(3)}$s plus one isolated vertex since any sunflower will automatically limit the number of edges to $n-2$. And if $n \equiv 2 \text{ mod } 4$ or $n \equiv 3 \text{ mod } 4$, then $n-2$ is the best that we can do.
\end{proof}

In general, the only way to actually have an $H_1$-free graph with $n$ multiedge triples is if the multiedge triples form an undirected 3-uniform hypergraph of $\frac{n}{4}$ disjoint $K_4^{(3)}$ blocks when $n \equiv 0 \text{ mod } 4$.

In this case there can be no additional directed edges in $H$ since such an edge would either intersect one of these $K_4^{(3)}$s in one tail vertex which would create a copy of $H_1$ since this means it intersects three of the multiedge triples in exactly one tail vertex (we may assume that each multiedge has all three edges per Lemma~\ref{more}) or it would intersect one of the $K_4^{(3)}$s in two tail vertices which means  that it intersects two of the multiedge triples in exactly one tail vertex (see Figure~\ref{blocks}).

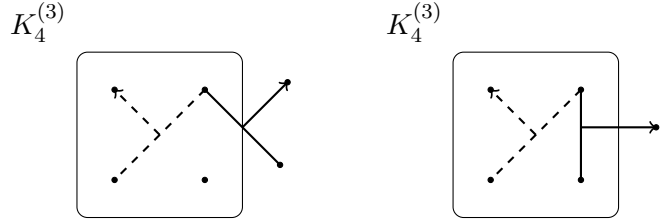
\begin{figure}
	\centering
	\begin{tikzpicture}
		\filldraw[black] (0.4,1.6) circle (1pt);
		\filldraw[black] (0.4,0.4) circle (1pt);
		\filldraw[black] (1.6,1.6) circle (1pt);
		\filldraw[black] (1.6,0.4) circle (1pt);
		\draw[rounded corners] (-0.1,-0.1) rectangle (2.1,2.1);
		\node[above left] at (-0.1,2.1) {$K_4^{(3)}$};
		
		\filldraw[black] (2.6,0.6) circle (1pt);
		\filldraw[black] (2.7,1.7) circle (1pt);
		\draw[thick] (1.6,1.6) -- (2.6,0.6);
		\draw[thick, ->] (2.1,1.1) -- (2.7,1.7);
		
		\draw[thick, dashed] (1.6,1.6) -- (0.4,0.4);
		\draw[thick, dashed,->] (1,1) -- (0.4,1.6);
		
		\filldraw[black] (5.4,1.6) circle (1pt);
		\filldraw[black] (5.4,0.4) circle (1pt);
		\filldraw[black] (6.6,1.6) circle (1pt);
		\filldraw[black] (6.6,0.4) circle (1pt);
		\draw[rounded corners] (4.9,-0.1) rectangle (7.1,2.1);
		\node[above left] at (4.9,2.1) {$K_4^{(3)}$};

		\filldraw[black] (7.6,1.1) circle (1pt);
		\draw[thick] (6.6,1.6) -- (6.6,0.4);
		\draw[thick, ->] (6.6,1.1) -- (7.6,1.1);
		
		\draw[thick, dashed] (6.6,1.6) -- (5.4,0.4);
		\draw[thick, dashed,->] (6,1) -- (5.4,1.6);
	\end{tikzpicture}
	\caption{An edge that intersects a $K_4^{(3)}$ block of multiedge triples in one or two tail vertices will create a copy of $H_1$.}
	\label{blocks}
\end{figure}

So in this case, the number of total edges would be bound by \[3n < {n+1 \choose 2} - 3\] for all $n \geq 7$.

Next, the only ways to have $n-1$ multiedge triples is to either have $\frac{n-1}{4}$ disjoint $K_4^{(3)}$ blocks when $n \equiv 1 \text{ mod } 4$ or to have $\frac{n}{4} - 1$ disjoint $K_4^{(3)}$ blocks with one $K_4^{-}$ when $n \equiv 0 \text{ mod } 4$. In the first case any additional edge must have at least one and perhaps two of its tail vertices in a single $K_4^{(3)}$ block of multiedge triples which we have already seen will create a copy of $H_1$. So there are at most \[3(n-1) < 3n < {n+1 \choose 2} - 3\] total edges in this case.

In the second case, any additional edge that has no tail vertices in a $K_4^{(3)}$ block must have both tail vertices in the $K_4^{-}$. If the head to such an edge were outside of the $K_4^{-}$, then the edge must intersect one of the three multiedge triples of the block in exactly one tail vertex since there are two triples that it intersects in one tail vertex each, one of which must be a multiedge triple. On the other hand, it could have its head vertex inside the $K_4^{-}$. In this case, the additional edge must lie on the triple without multiple edges. This is the only edge that can be added so there are at most \[3(n-1)+1 <3n <{n+1 \choose 2} - 3\] total edges in this case.

\subsection{An $H_1$-free graph with $n-2$ multiedge triples}

Now, the only ways to have exactly $n-2$ multiedge triples is either to have $\frac{n}{4} - 2$ of the $K_4^{(3)}$ blocks plus two $K_4^{-}$ blocks of multiedge triples when $n \equiv 0 \text{ mod } 4$ or to have $k$ of the $K_4^{(3)}$ blocks of multiedge triples plus a sunflower with $n-4k-2$ petals. The first case is suboptimal for the same reasons already considered. So let's consider the second case.

First, assume that $k=0$ and that we have $n-2$ multiedge triples that make a sunflower (see Figure~\ref{mexconst}). How many edges can we add? This structure already has all possible edges with 2 vertices in the core (or so we may assume by Lemma~\ref{more}). On the other hand, if an additional edge has no vertices in the core, then it would intersect two multiedge triples in one tail vertex each which would create a copy of $H_1$.

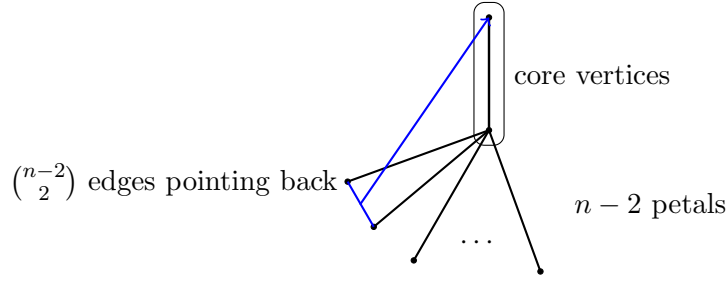
\begin{figure}
	\centering
	\begin{tikzpicture}
		\filldraw[black] (0,0) circle (1pt);
		\filldraw[black] (0,1.5) circle (1pt);
		\filldraw[black] (xyz polar cs:angle=200, radius=2) circle (1pt);
		\filldraw[black] (xyz polar cs:angle=220, radius=2) circle (1pt);
		\filldraw[black] (xyz polar cs:angle=240, radius=2) circle (1pt);
		\filldraw[black] (xyz polar cs:angle=290, radius=2) circle (1pt);
		
		\draw[thick] (0,1.5) -- (0,0) -- (xyz polar cs:angle=200, radius=2);
		\draw[thick] (0,1.5) -- (0,0) -- (xyz polar cs:angle=220, radius=2);
		\draw[thick] (0,1.5) -- (0,0) -- (xyz polar cs:angle=240, radius=2);
		\draw[thick] (0,1.5) -- (0,0) -- (xyz polar cs:angle=290, radius=2);
		
		\draw[rounded corners] (-0.2,-0.2) -- (0.2,-0.2) -- (0.2,1.7) -- (-0.2,1.7) -- cycle;
		
		\node at (xyz polar cs:angle=265, radius=1.5) {$\cdots$};
		
		\node [right] at (0.2,0.75) {core vertices};
		\node [right] at (1,-1) {$n-2$ petals};
		\node [left] at (xyz polar cs:angle=200, radius=2) {${n-2 \choose 2}$ edges pointing back};
		
		\draw[color=blue, thick] (xyz polar cs:angle=200, radius=2) -- (xyz polar cs:angle=220, radius=2);
		\draw[color=blue, thick,->] (xyz polar cs:angle=210, radius=1.98) -- (0,1.5);
		
	\end{tikzpicture}
	\caption{The unique extremal construction for an $H_1$-free graph has ${n-2 \choose 2} + 3(n-2)$ edges.}
	\label{mexconst}
\end{figure}

Therefore, any additional edge must include exactly one vertex from the core. If this vertex is in the tail set to the additional edge and the sunflower has at least three petals, then the additional edge intersects in exactly one tail vertex one of the multiedge triples of the sunflower, a contradiction. Since we assume that $n \geq 6$, then the sunflower has at least three petals. Hence, any additional edge must intersect the core in only its head vertex.

If any two additional edges have different core vertices as the head, then either the tails sets are the same or completely disjoint to avoid a copy of $H_1$. Hence, pairs of petal vertices that point to both core vertices must be independent of all other tails sets. And all other petal vertices fall into disjoint sets as to whether they are in additional edges that point to the first core vertex or the second. The number of additional edges will be maximized if every pair of petal vertices point to the same core vertex. Moreover, this will give a total of \[3(n-2) + {n-2 \choose 2} = {n+1 \choose 2} - 3\] edges.

We will soon see that this is the best that we can do and that this construction, where the multiedge triples make a sunflower with $n-2$ petals with ${n-2 \choose 2}$ additional edges pointing from pairs of petal vertices to a single core vertex, is unique up to isomorphism.

First we will need to see that $k=0$ is the number of $K_4^{(3)}$ multiedge triple blocks that optimizes the total number of edges. So suppose there are $k$ such blocks and that the other $n-4k$ vertices are in a sunflower. Then from prior considerations we know that any additional edge must have both tail vertices in this sunflower. If one of these tail vertices coincides with a petal vertex of the sunflower, then there will be a copy of $H_1$. Therefore, the tail vertices must coincide with the core and the only possibility for such an edge is to point out to a vertex in one of the $k$ blocks.

Therefore, there are at most \[3(4k) + 3(n-4k-2) + {n-4k-2 \choose 2} + 4k\] edges in such a construction. Since this expression is quadratic in $k$ with positive leading coefficient, then it must maximize at the endpoints, $k=0$ or $k=\frac{n}{4}$, and we already know that $k=\frac{n}{4}$ is suboptimal. Therefore, if there are exactly $n-2$ multiedge triples, then they must form a sunflower with a two-vertex core and from there the only way to maximize the total number of edges is to add every possible edge with tail set among the petal vertices all pointing to the same head vertex in the core.

\subsection{Fewer than $n-2$ multiedge triples}

Now suppose that $H$ has fewer than $n-2$ multiedge triples. If $t_H \leq {n-1 \choose 2}$, then \[|E(H)| \leq t_H + 2m_H < {n-1 \choose 2} + 2(n-2) = {n+1 \choose 2} - 3.\] So we must assume that $t_H > {n-1 \choose 2}$. Also, if $m_H=0$, then we know that \[|E(H)| \leq \text{ex}_o(n,H_1) = \left\lfloor \frac{n}{2} \right\rfloor (n-2) < {n+1 \choose 2} - 3.\] So assume that there is at least one multiedge triple, $\{x,y,z\}$. This triple has at least two edges. Assume without loss of generality that they are $xy \rightarrow z$ and $xz \rightarrow y$.

Let $H'$ be an oriented graph arrived at by deleting edges from multiedge triples of $H$ until each triple has at most one edge and every triple that had at least one edge in $H$ still has at least one in $H'$. In other words, $H'$ is any subgraph of $H$ such that $t_{H'} = t_{H}$ and $m_{H'}=0$. Without loss of generality, assume that \[xy \rightarrow z \in E(H').\]

Since $t_{H'} > {n-1 \choose 2}$, then $n$ must be even, and moreover, there is a matching on the vertices so that every matched pair $\{a,b\}$ points to at least two other vertices. That is, $t(a,b) \geq 2$.

Now consider the directed link graphs of the vertices. As stated before, these are either triangles or stars with a common vertex. However, if two or more of these link digraphs have three or fewer edges each (for instance, if they are triangles), then there are fewer edges than we are assuming since \[|E(H')| = \frac{1}{2} \sum_{x \in V(H')} |D_x| \leq \frac{1}{2} \left(6 + (n-3)(n-2) \right) < {n-1 \choose 2}\] for all $n \geq 8$. We will show that it must be the case that here at least two directed link graphs are restricted to at most three directed edges each, contradicting our assumptions about the number of edges in $H$.

First, note that $x \rightarrow z \in D_y$ and $y \rightarrow z \in D_x$. To avoid a contradiction, at least one of these two directed link graphs must have four or more edges. Without loss of generality, assume that it is $D_y$. Therefore, $D_y$ is a star and not a triangle. So the additional three directed edges in $D_y$ must either all be incident to $z$ or to $x$.

If these directed edges are all incident to $z$, then $y$ and $z$ must be partners under the matching which means that $x$ has another partner $x'$ distinct from $y$ and $z$. Since $t(x,x') \geq 2$ in $H'$, then $x'$ must point to two vertices in $D_x$. Since $D_x$ already has $y \rightarrow z$ and no two edges may be independent in any directed link graph, then $x'$ must point to $y$ and to $z$, forming a triangle.

Next, consider $D_{x'}$. We know that \[x \rightarrow y, x \rightarrow z \in D_{x'}.\] If there is an additional edge in $D_{x'}$ that does not complete this triangle then it is either of the form $x \rightarrow t$ or $t \rightarrow x$. If $x \rightarrow t \in D_{x'}$ then $x' \rightarrow t, y \rightarrow z \in D_x$, a contradiction. If $t \rightarrow x \in D_{x'}$, then $x' \rightarrow x \in D_t$. But since $t$ has its own matched vertex, then there exists a distinct $t'$ such that \[t' \rightarrow x, t' \rightarrow x' \in D_{t'}.\] So either $|D_{x'}| \leq 3$ or $|D_{t'}| \leq 3$. Either way, this gives us two directed link graphs that have at most three edges each. So $t_H > {n-1 \choose 2}$.

Therefore, we must assume that the three additional edges in $D_y$ are incident to $x$ and that $y$ and $x$ are partners under the matching. So $z$ has some other partner under the matching $z'$ distinct from $x$ and $y$. Now, delete the edge $xy \rightarrow z$ from $H'$ and add $xz \rightarrow y$ to get a new directed hypergraph $H''$. It follows that $H''$ has no multiedge triples and is $H_1$-free since we still have a subgraph of $H$.

In adding $xz \rightarrow y$ we have added $x \rightarrow y$ to $D_z$. Since $z'$ must point to two vertices in $D_z$, then this addition means that $D_z$ is a triangle under $H''$. Hence, $|D_z| = 2$ under $H'$.

Now, the same argument as above applies to $D_{z'}$. The only way for $|D_{z'}| > 3$ would mean either $z \rightarrow a \in D_{z'}$ or $a \rightarrow z \in D_{z'}$ for some $a$ distinct from $x$, $y$, $z$, and $z'$. The first case would mean that two independent directed edges, $z' \rightarrow a$ and $x \rightarrow y$ are in $D_z$, a contradiction. The second case would mean that $z' \rightarrow z \in D_a$. Since $a$ has its own partner under the matching that must point to two vertices in $D_a$, then in this case, $D_a$ is a triangle.

Therefore, $t_H > {n-1 \choose 2}$ and $m_H \geq 1$ cannot both be true in any $H_1$-free graph. This is enough to complete the result, \[\text{ex}(n,H_1) = {n+1 \choose 2} - 3.\] This also exhausts the remaining cases in order to demonstrate that the extremal construction is unique.

\section{Forbidden $H_2$}

In this section $H_2$ denotes the forbidden graph where two edges intersect in exactly two vertices such that the set of intersection is the tail set to each edge. That is $V(H_2) = \{a,b,c,d\}$ and $E(H_2) = \{ab \rightarrow c, ab \rightarrow d\}$ (see Figure~\ref{D}).

\begin{theorem}
For all $n \geq 5$, \[\text{ex}(n,H_2) = \text{ex}_o(n,H_2) = {n \choose 2}.\] Moreover, there are ${n \choose 2}^{n-2}$ different labeled $H_2$-free graphs attaining this extremal number when in the standard version of the problem.
\end{theorem}

\begin{figure}
	\centering
	\begin{tikzpicture}
		\filldraw[black] (0,1) circle (1pt);
		\filldraw[black] (0,-1) circle (1pt);
		\filldraw[black] (-2,0) circle (1pt);
		\filldraw[black] (2,0) circle (1pt);
		\draw[thick] (0,1) -- (0,-1);
		\draw[thick, <->] (-2,0) -- (2,0);
		
		\node[left] at (-2,0) {$c$};
		\node[right] at (2,0) {$d$};
		\node[above] at (0,1) {$a$};
		\node[below] at (0,-1) {$b$};
	\end{tikzpicture}
	\caption{$H_2$}
	\label{D}
\end{figure}

\begin{proof}
Let $H$ be $H_2$-free. Regardless of which version of the problem we are considering, each pair of vertices acts as the tail set to at most one directed edge. Therefore, \[\text{ex}(n,H_2), \text{ex}_o(n,H_2) \leq {n \choose 2}.\]

In the standard version of the problem any function, $f: {[n] \choose 2} \rightarrow [n]$, that sends each pair of vertices to a distinct third vertex, $f(\{a,b\}) \not \in \{ a,b \}$, has an associated  $H_2$-free construction $H_f$ with ${n \choose 2}$ edges. That is, for any such function, $f$, let $V(H_f) = [n]$ and \[E(H_f) = \left\{a,b \rightarrow f(\{a,b\}) : \{a,b\} \in {[n] \choose 2}\right\}.\] Since each pair of vertices acts as the tail set to exactly one directed edge, then $H_f$ is $H_2$-free and has ${n \choose 2}$ edges. So \[\text{ex}(n,H_2) = {n \choose 2}.\]

Moreover, there are ${n \choose 2}^{n-2}$ distinct functions from ${[n] \choose 2}$ to $[n]$ such that no pair is mapped to one of its members. Therefore, there are ${n \choose 2}^{n-2}$ labeled graphs that are $H_2$-free with ${n \choose 2}$ edges.\\

In the oriented version of the problem lower bound constructions can be defined inductively on $n$.

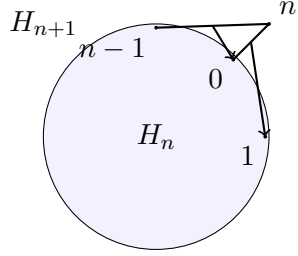
\begin{figure}
	\centering
	\begin{tikzpicture} [scale=0.5]
		\node at (-3,3) {$H_{n+1}$};
		\filldraw[color=black,fill=blue!5] (0,0) circle [radius=3];
		\node at (0,0) {$H_n$};
		
		\filldraw[black] (3,3) circle (1pt);
		\node[above right] at (3,3) {$n$};
		
		\filldraw[black] (0,2.9) circle (1pt);
		\node[below left] at (0,2.9) {$n-1$};
		\filldraw[black] (2.05,2.05) circle (1pt);
		\node[below left] at (2.05,2.05) {$0$};
		\filldraw[black] (2.9,0) circle (1pt);
		\node[below left] at (2.9,0) {$1$};
		
		\draw[thick] (3,3) -- (0,2.9);
		\draw[thick] (3,3) -- (2.05,2.05);
		
		\draw[thick,->] (1.5,2.95) -- (2.05,2.05);
		\draw[thick,->] (2.52,2.52) -- (2.9,0);
	\end{tikzpicture}
	\caption{Inductive construction of $H_2$-free oriented graphs}
	\label{D-free}
\end{figure}

First, let $n=5$ and define $G_5$ as the oriented graph with vertex set \[V(G_5) = \{0,1,2,3,4\}\]and the following edges: $0,1 \rightarrow 2$, $1,3 \rightarrow 0$, $0,4 \rightarrow 1$, $0,2 \rightarrow 3$, $2,4 \rightarrow 0$, $0,3 \rightarrow 4$, $2,3 \rightarrow 1$, $1,2 \rightarrow 4$, $1,4 \rightarrow 3$, and $3,4 \rightarrow 2$.

Each pair of vertices of $G_5$ are in exactly one tail set, and each triple of vertices appear together in exactly one edge. Therefore, this construction is $H_2$-free with ${5 \choose 2}$ edges.

Now, let $n > 5$, and define $G_n$ by $V(G_n) = \mathbb{Z}_n$ and \[E(G_n) = E(G_{n-1}) \cup \{(n-1)i \rightarrow (i+1) : i = 0,\ldots,n-2\}.\]

Then $G_n$ has $n-1$ more edges than $G_{n-1}$. So $|E(H_n)| = {n \choose 2}$.

Any two new edges intersect in at most two vertices. Similarly, any new edge and any old edge also intersect in at most two vertices. Hence, at most one edge appears on a given triple of vertices. So each $G_n$ is oriented.

Moreover, all tail sets for the new edges are distinct from each other and from any tail sets for the edges of $G_{n-1}$. So $G_n$ is $H_2$-free. Therefore, \[\text{ex}_o(n,H_2) = {n \choose 2}.\]
\end{proof}

\section{Conclusion}

The $2 \rightarrow 1$ version of directed hypergraph originally came to the author's attention as a way to model definite Horn clauses in propositional logic. Definite Horn clauses are more generally modeled by $r \rightarrow 1$ edges for any $r$. Therefore, it seems natural to ask about the extremal numbers for graphs with two $(r \rightarrow 1)$-edges. If we look at every $(r \rightarrow 1)$-graph with exactly two edges we see that these fall into four main types of graph. Let $i$ be the number of vertices that are in the tail set of both edges. Then let $I_r(i)$ denote the graph where both edges point to the same head vertex, let $H_r(i)$ denote the graph where the edges point to different head vertices neither of which are in the tail set of the other, let $R_r(i)$ denote the graph where the first edge points to a head vertex in the tail set of the second edge and the second edge points to a head not in the tail set of the first edge, and let $E_r(i)$ denote the graph where both edges point to heads in the tail sets of each other. This extends the notation used in this paper.

The degenerate cases here would be $I_r(i)$ and $H_r(i)$. It would be interesting to find the extremal numbers for these graphs in general. To what extent do the current proofs extend to these graphs? For example, in the standard version of the problem it can easily be seen that \[\text{ex}(n,I_r(0)) = n{n-2 \choose r-1}\] using the same idea as we did for $I_0$. Will the other ideas generalize as well?

\bibliography{degenerate}
\bibliographystyle{plain}

\end{document}